\documentclass[11pt,reqno]{amsart}
\usepackage[utf8]{inputenc}

\usepackage[makeroom]{cancel}
\usepackage{float}
\usepackage{mathabx}
\usepackage{amsmath}
\usepackage{amsthm}
\usepackage{amssymb}
\usepackage{amscd}
\usepackage{amsfonts}
\usepackage{amsbsy}
\usepackage{graphicx}
\usepackage[dvips]{psfrag}
\usepackage{array}
\usepackage{xcolor}
\usepackage{epsfig}
\usepackage{url}
\usepackage{hyperref}
\usepackage{enumitem}
\usepackage{overpic}
\usepackage{epstopdf}
\usepackage[top=3cm,bottom=2cm,left=1.6cm,right=1.6cm]{geometry}

\newcommand{\R}{\ensuremath{\mathbb{R}}}

\newcommand{\CO}{\ensuremath{\mathcal{O}}}

\newcommand{\D}{\ensuremath{\mathcal{D}}}

\newcommand{\x}{\mathbf{x}}

\newcommand{\bb}{\mathbf{b}}
\newcommand{\sgn}{\mathrm{sign}}

\newcommand{\Int}{\mathrm{Int}}
\newcommand{\QQ}{\Theta_{\rm{\i v}}}

\newtheorem {theorem} {Theorem}

\newtheorem {proposition}{Proposition}
\newtheorem {corollary} {Corollary}
\newtheorem {lemma}  {Lemma}

\newtheorem {remark}{Remark}

\usepackage{apptools}
\AtAppendix{\counterwithin{theorem}{section}}

\newtheorem{main}{Theorem}

\begin{document}
\renewcommand{\arraystretch}{1.5}

\title[Uniqueness and stability of limit cycles of sewing PWLS]
{Uniqueness and stability of limit cycles in planar piecewise linear differential systems without sliding region}

\author[V. Carmona, Fern\'{a}ndez-S\'{a}nchez, and D. D. Novaes]
{Victoriano Carmona$^1,$ Fernando Fern\'{a}ndez-S\'{a}nchez$^2,$\\ and Douglas D. Novaes$^3$}

\address{$^1$ Dpto. Matem\'{a}tica Aplicada II \& IMUS, Universidad de Sevilla, Escuela Polit\'ecnica Superior.
Calle Virgen de \'Africa 7, 41011 Sevilla, Spain.  Corresponding author.}
\email{vcarmona@us.es}

\address{$^2$ Dpto. Matem\'{a}tica Aplicada II \& IMUS, Universidad de Sevilla, Escuela T\'{e}cnica Superior de Ingenier\'{i}a.
Camino de los Descubrimientos s/n, 41092 Sevilla, Spain.} 
\email{fefesan@us.es} 

\address{$^3$Departamento de Matem\'{a}tica, Instituto de Matem\'{a}tica, Estat\'{i}stica e Computa\c{c}\~{a}o Cient\'{i}fica (IMECC), Universidade
Estadual de Campinas (UNICAMP), Rua S\'{e}rgio Buarque de Holanda, 651, Cidade Universit\'{a}ria Zeferino Vaz, 13083--859, Campinas, SP,
Brazil.} \email{ddnovaes@unicamp.br} 



\begin{abstract}
In this paper, we consider the family of planar piecewise linear differential systems with two zones separated by a straight line without sliding regions, that is, differential systems whose flow transversally crosses the switching line except for at most one point. In the research literature, many papers deal with the problem of determining the maximum number of limit cycles that these differential systems can have. This problem has been usually approached via large case-by-case analyses which distinguish the many different possibilities for the spectra of the matrices of the differential systems. Here, by using a novel integral characterization of Poincar\'e half-maps, we prove, without unnecessary distinctions of matrix spectra, that the optimal uniform upper bound for the number of limit cycles of these differential systems is one. In addition, it is proven that this limit cycle, if it exists, is hyperbolic and its stability is determined by a simple condition in terms of the parameters of the system.  As a byproduct of our analysis, a condition for the existence of the limit cycle is also derived.
\end{abstract}

\maketitle

\section{Introduction}
In this paper, we consider planar discontinuous piecewise linear differential systems with two zones separated by a straight line. Without loss of generality, these differential systems can be written as
\begin{equation}\label{s1}
\dot \x =
\left\{\begin{array}{l}
A_L\x+\bb_L, \quad\textrm{if}\quad x_1< 0,\\
A_R\x+\bb_R, \quad\textrm{if}\quad x_1> 0,
\end{array}\right.
\end{equation}
where $\x=(x_1,x_2)^T\in\R^2,$ $A_{L}=(a_{ij}^{L})_{2\times 2},$ $A_{R}=(a_{ij}^{R})_{2\times 2},$  $\bb_{L}=(b_1^{L},b_2^{L})^T\in\R^2,$ and $\bb_{R}=(b_1^{R},b_2^{R})^T\in\R^2.$ Notice that $\Sigma=\{(x_1,x_2)\in\R^2:\,x_1=0\}$ the switching line of system \eqref{s1}. We assume Filippov's convention (see \cite{Filippov88}) for the definition of the trajectories of \eqref{s1}. 

Moreover, we are interested in the family of such differential systems which do not have any sliding regions, that is, for each point in the switching line, except for at most one, the corresponding trajectory crosses it transversally. Such systems are sometimes called {\em sewing systems}. This can be analytically expressed by
\begin{equation}
\label{ecu:parabola}
(a^L_{12} \, x_2+b_1^L)(a^R_{12}\, x_2+b_1^R)\geqslant 0 \quad \mbox{for every} \quad x_2\in\mathbb{R}
\end{equation}
and 
\begin{equation}
\label{ecu:tangencypoint}
(a^L_{12} \, x_2+b_1^L)(a^R_{12}\, x_2+b_1^R)=0 \quad \mbox{at most at one value of} \quad x_2\in\mathbb{R}.
\end{equation}
Notice that, in this paper, the concept of  ``sliding region'' also includes the sometimes called in other papers ``escaping region''. 

We focus on determining an optimal uniform upper bound for the number of limit cycles of this family of differential systems. In the study of planar piecewise smooth differential systems, a  limit cycle is usually defined as a non-trivial closed crossing orbit which is isolated from other closed orbits.

In 1991, Lum and Chua \cite{LumChua91}, assuming the continuity of the differential system \eqref{s1} conjectured that they have at most one limit cycle. This conjecture was proven in 1998 by Freire et al. \cite{FreireEtAl98}.  Their proof was performed by distinguishing every possible configuration depending on the spectra of the matrices of the differential system. 
Recently, Carmona et al. \cite{CARMONA2021100992}  provided a new simple proof for Lum-Chua's conjecture using a {\it novel integral characterization} of {\it Poincar\'{e} half-maps}  (see \cite{CarmonaEtAl19}). In addition, this limit cycle, if it exists, has been proven to be hyperbolic and its stability was explicitly determined by an easy condition in terms of the parameters of the system. This novel characterization has proven to be an effective method to avoid the case-by-case study performed in the former proof and has also been used by the same authors in \cite{CarmonaEtAl23} to show the existence of a uniform upper bound for the maximum number of limit cycles of general planar piecewise linear differential systems with two zones separated by a straight line.

Dropping the continuity assumption, 
Freire et al. \cite{FreireEtAl2013} studied the limit cycles of the differential system \eqref{s1}  
assuming that each linear differential system has a real or virtual equilibrium of focus type 
(the concepts of {\em real}, {\em boundary} or {\em virtual},  referred to an equilibrium of a system corresponding to one of the zones of linearity, just describe if it is located, respectively, inside, on the boundary or outside this zone).
In this paper, the authors qualitatively described the different phase portraits taking into account the number of real focus equilibria, namely zero, one, or two. For the cases of zero and one real focus equilibrium, they obtained results on the existence and uniqueness of limit cycles. For the case of two real focus equilibria, based on extensive numerical simulations, they conjectured that such differential systems have at most one limit cycle. The case of two virtual foci had already been considered by Llibre et al. \cite{LlibreEtAl2008}, who obtained the uniqueness of limit cycles via a generalized criterium for Li\'{e}nard differential equations allowing discontinuities.  
The uniqueness of limit cycles for the node-node and saddle-saddle cases have been addressed by Huan and Yang in  \cite{HuanYang2013} and \cite{HuanYang2014}, respectively. Medrado and Torregrosa \cite{MedradoTorregrosa15} provided the uniqueness of limit cycles  assuming the existence of a monodromic singularity in the switching line. 
Their proof also  distinguishes some configurations depending on the spectra of matrices of the differential system. Li and Llibre \cite{LiLLibre2020} gave a proof of the uniqueness of limit cycles for the focus-saddle case.  Recently, Li et al. \cite{LI2021101045} provided the uniqueness of limit cycles for the focus-node and focus-focus scenarios, which were the remaining non-degenerate cases (that is, $\det(A_L)\det(A_R)\neq0$). By means of a different approach, Tao Li et al. \cite{Li_2020} also provided the uniqueness of limit cycles for the non-degenerate cases. These last two papers provided a positive answer for the conjecture stated in \cite{FreireEtAl2013} for the case of two real focus equilibria. However, taking into account all the previous results, a positive answer to the uniqueness of limit cycles of planar piecewise linear differential systems without sliding region was only obtained for the non-degenerate cases. As far as we are concerned in all previous case-by-case studies, the degenerate cases have not been exhausted. It is worth noting all the effort and time needed, the large number of papers devoted to different cases, and the specific techniques developed in each one of them to reach the result for the non-degenerate cases. 

In this paper we provide the uniqueness of limit cycles for piecewise linear differential systems without sliding region. Moreover, this result is proven in a unified way (with a single approach and a common technique) that does not distinguish cases depending on the spectrum of the matrices. This unified way allows us to establish the hyperbolicity of the limit cycle and also to determine its stability by a simple condition in terms of the parameters of system  \eqref{s1}. These results are collected in the statement of the main theorem of this paper. 

\begin{main}\label{main} 
Consider the planar piecewise linear differential system \eqref{s1}. Let $T_L$ and $T_R$ be the traces of the matrices $A_L$ and $A_R$, respectively. Denote
$a_{L}=a_{12}^{L}b_2^{L}-a_{22}^{L}b_1^{L}$, $a_{R}=a_{12}^{R}b_2^{R}-a_{22}^{R}b_1^{R}$, and $\xi=a_RT_L-a_LT_R$. If the differential system \eqref{s1} does not have sliding region (that is, conditions \eqref{ecu:parabola} and \eqref{ecu:tangencypoint} hold), then it has at most one limit cycle. This limit cycle, if it exists, is hyperbolic and $\xi\neq0.$ Moreover, it is asymptotically stable (resp. unstable) provided $\xi<0$  (resp. $\xi>0$).
\end{main}

The proof of Theorem \ref{main} 
is a direct consequence of propositions \ref{prop:lienard}, \ref{prop:hyperlimitcycle}, and \ref{prop:dlc}, as follows: Proposition \ref{prop:lienard} recalls the normal form \eqref{cf} for  piecewise linear differential systems given by \eqref{s1}; then, Proposition \ref{prop:hyperlimitcycle} provides the uniqueness of hyperbolic limit cycles and characterizes its stability for piecewise linear differential systems without sliding region; lastly, Proposition \ref{prop:dlc} concludes that such systems do not admit non-hyperbolic limit cycles.

The main basis of this paper lies in an original integral characterization, presented in \cite{CarmonaEtAl19}, of the {\it Poincar\'{e} half-maps} for planar linear differential systems associated to a straight line. This novel  characterization 
is introduced in Section \ref{sec:properties}, where some properties of these half-maps, interesting for our study, are collected. As usual, the study of the crossing limit cycles for a two-zonal piecewise linear differential system 
 is done by means of the analysis of zeros of an appropriate displacement function given as the difference between the Poincar\'{e} half-maps associated to the switching line. Section \ref{sec:displfunct} is devoted to introduce and analyse the displacement function by using the results stated in Section \ref{sec:properties}. More specifically, we provide suitable expressions for its first and second derivatives. The local behaviour of the displacement function around monodromic singularities on $\Sigma$ and, when it is appropriate, at the infinity is established in Section \ref{sec:stability}. 
Some results on the existence of limit cycles are given in Section \ref{sec:condlimitcycle}. In particular,
 Corollary \ref{cor:existence} provides an extension to sewing piecewise linear differential systems of the results about existence of limit cycles for continuous differential systems given in [5]. In Section \ref{sec:hyperlimitcycle}, we show the uniqueness of hyperbolic limit cycles (see Proposition \ref{prop:hyperlimitcycle}) for piecewise linear differential systems without sliding regions. Moreover, the stability of this unique limit cycles,  if it exists, is determined by a simple condition in terms of the parameters of the system. Finally, Section \ref{sec:dlc} is dedicated to showing  that  the considered differential systems do not admit degenerate limit cycles (see Proposition \ref{prop:dlc}).

\section{Li\'enard canonical form and Poincar\'{e} half-maps
}\label{sec:properties}
First of all, notice that if $a_{12}^{L}=0$, then the first equation of system \eqref{s1} for $x_1<0$ becomes $\dot x_1=a_{11}^L x_1+b^L_1$, whose solutions are all monotonic and prevents the existence of periodic solutions for system \eqref{s1}. With a similar reasoning, $a_{12}^{R}=0$ also prevents the existence of limit cycles for system \eqref{s1}. Thus, the condition $a_{12}^{L}a_{12}^{R}\ne 0$ is necessary for the existence of periodic solutions of system \eqref{s1}. Under this hypothesis, the non-sliding conditions \eqref{ecu:parabola} and \eqref{ecu:tangencypoint} are equivalent to 
\begin{equation}
\label{eq:nslc}
  a^L_{12} a^R_{12}>0 \quad \mbox{and} \quad a_{12}^L \, b_1^R=a_{12}^R \, b_1^L.
\end{equation}
From now on, we assume that the differential system \eqref{s1} satisfies condition \eqref{eq:nslc}.

A natural first step in the analysis of any parameterized differential system consists in writing it in a suitable normal form. The following proposition, whose proof appeared in \cite{FreireEtAl12} for a more general case, is devoted to it. 
\begin{proposition}
\label{prop:lienard}
Under assumption \eqref{eq:nslc}, the differential system \eqref{s1} is reduced into the following Li\'enard canonical form 
\begin{equation}\label{cf}
\left\{\begin{array}{l}
\dot x= T_L x-y\\
\dot y= D_L x-a_L
\end{array}\right.\quad \text{for}\quad x< 0,
\quad 
\left\{\begin{array}{l}
\dot x= T_R x-y\\
\dot y= D_R x-a_R
\end{array}\right.\quad \text{for}\quad x> 0,
\end{equation}
by a homeomorphism preserving the switching line $\Sigma=\{(x,y)\in\R^2:\,x=0\}$. Furthermore, $a_{L}=a_{12}^{L}b_2^{L}-a_{22}^{L}b_1^{L},$ $a_{R}=a_{12}^{R}b_2^{R}-a_{22}^{R}b_1^{R},$ and $T_L,$ $T_R$ and $D_L,$ $D_R$ are, respectively, the traces and determinants of the matrices $A_L$ and $A_R$.
\end{proposition}

As said before, the study of the crossing limit cycles for the differential system \eqref{cf}
 is done by means of the analysis of the zeros of an appropriate displacement function given as the difference between the Poincar\'{e} half-maps associated to the switching line $\Sigma$. In what follows we introduce such maps, namely, the {\it Forward Poincar\'{e} Half-Map}  $y_L: I_L\subset [0,+\infty) \longrightarrow(-\infty,0]$ and  the {\it Backward Poincar\'{e} Half-Map} $y_R:I_R\subset [0,+\infty)\rightarrow (-\infty,0]$, whose graphs are contained in the fourth quadrant
 \[
\QQ:=\{(y_0,y_1)\in\mathbb{R}^2:\,y_0\geq0,\, y_1\leq0\}.
 \]

The forward Poincar\'{e} half-map takes a point $(0,y_0)$, with $y_0\geq0$, and maps it to a point $(0,y_L(y_0))$ by following the forward flow of system \eqref{cf}. More specifically, 
let $\varphi(t;y_0)=(\varphi_1(t;y_0),\varphi_2(t;y_0))$ be the orbit of system \eqref{cf} satisfying $\varphi(0;y_0)=(0,y_0)$. If there exists a value $\tau_L(y_0)>0$ such that $\varphi_1(\tau_L(y_0);y_0)=0$ and $\varphi_1(t;y_0)<0$ for every $t\in(0,\tau_L(y_0))$, we define  $y_L(y_0)=\varphi_2(\tau_L(y_0);y_0)\leqslant0$ (see Figure \ref{fig:poincare}(a)). In addition, if for every $\varepsilon>0$ there exist $y_0\in(0,\varepsilon)$ and $y_1\in(-\varepsilon,0)$ such that $y_L(y_0)=y_1$, the left Poincar\'e half-map can be extended to $y_0=0$ with $y_L(0)=0$, even if the above positive time $\tau_L(0)$ does not exist (see Figure \ref{fig:poincare}(b)).

\begin{figure}[H]
    \begin{center}
    \begin{tabular}{cc}
     \includegraphics[width=5cm]{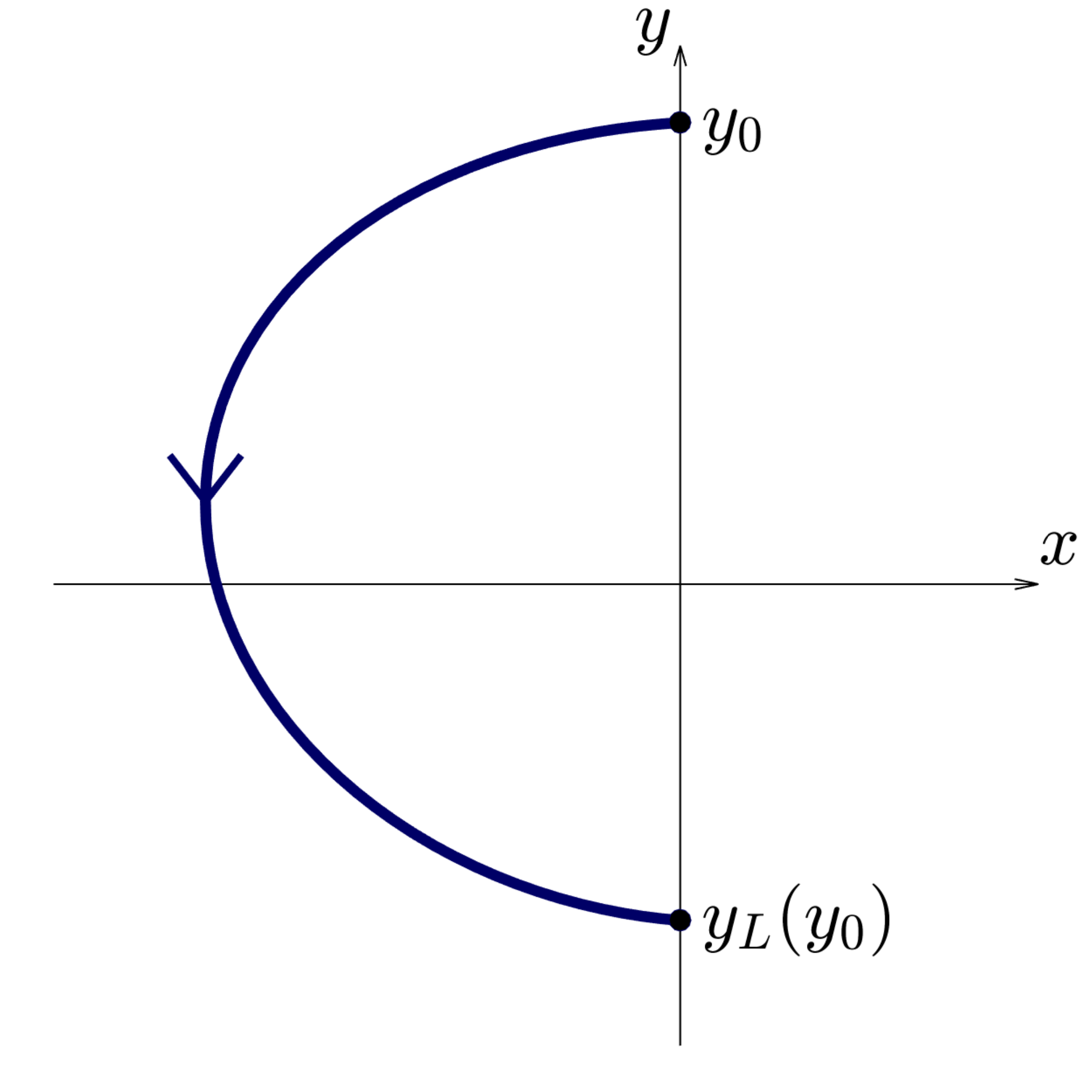}& \hspace{0.5cm}
     \includegraphics[width=5cm]{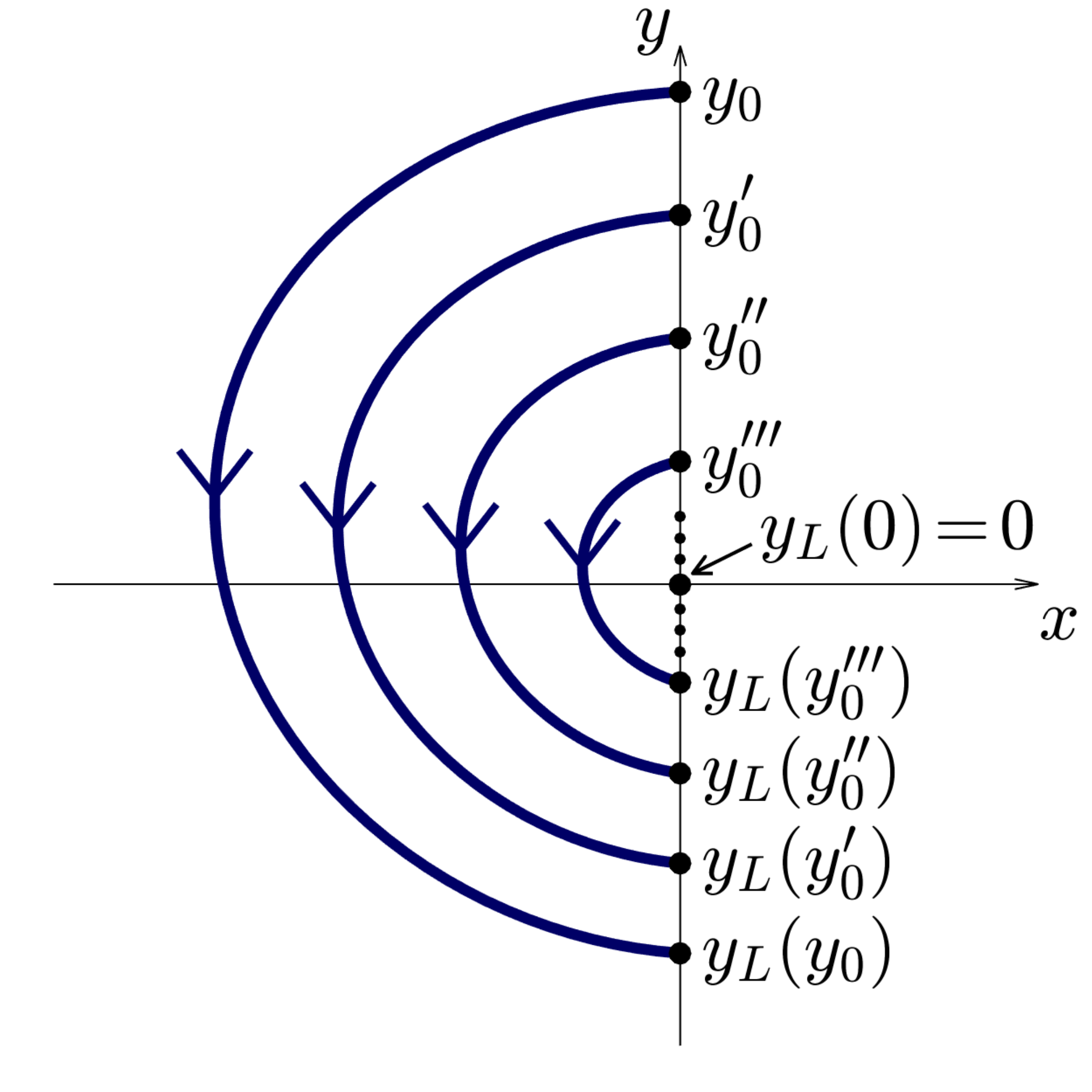}\\ 
     {\footnotesize (a)} & \hspace{0.5cm} {\footnotesize (b)}
     \end{tabular}
    \end{center}
     \caption{Definition of the forward Poincar\'{e} half-map.}\label{fig:poincare}
\end{figure}

Analogously, the backward Poincar\'{e} half-map takes a point $(0,y_0)$, with $y_0\geq 0$, and maps it to a point $(0,y_R(y_0))$ by following the backward flow of \eqref{cf}. More specifically, if there exists a value $\tau_R(y_0)<0$ such that $\varphi_1(\tau_R(y_0);y_0)=0$ and $\varphi_1(t;y_0)>0$ for every $t\in(\tau_R(y_0),0)$, we define  $y_R(y_0)=\varphi_2(\tau_R(y_0);y_0)\leqslant0$. Again, if for every $\varepsilon>0$ there exist $y_0\in(0,\varepsilon)$ and $y_1\in(-\varepsilon,0)$ such that $y_R(y_0)=y_1$, the right Poincar\'e half-map can be extended with $y_R(0)=0$, even if the above negative time $\tau_R(0)$ does not exist.

In Figure \ref{fig:displacement}, we illustrate the Poincar\'{e} half-maps defined above and their intervals of definition when the left linear system has a real focus and the right linear system has a real saddle.

In the next theorems \ref{teo:defyL} and \ref{teo:defyR}, we will present an integral characterization of the Poincar\'{e} half-maps above which was introduced in \cite{CarmonaEtAl19}. For that, we will need the following concept of {\it Cauchy Principal Value}:
 \[
\operatorname{PV}\left\{\int_{y_1}^{y_0}f(y)dy\right\}:=\lim_{\varepsilon\searrow 0} \left(\int_{y_1}^{-\varepsilon}f(y)dy+\int_{\varepsilon}^{y_0}f(y)dy\right),
\]
 for $y_1<0<y_0$ and $f$  continuous in $[y_1,y_0]\setminus \{0\}$ (see, for instance, \cite{henrici}). Note that if $f$ is also continuous in $0$, then the Cauchy Principal Value coincides with the definite integral.

Since the flow of \eqref{cf} is oriented in anticlockwise direction, the forward Poincar\'{e} half-map, $y_L$, is determined by the following linear differential system
\begin{equation}
\label{ecu:linealL}
\left\{\begin{array}{l}
\dot x= T_L x-y,\\
\dot y= D_L x-a_L,
\end{array}\right.
\end{equation}
which matches the left linear system of \eqref{cf}. Accordingly, its definition, its domain $I_L$, and its analyticity are given by Theorem 19, Corollary 21, Corollary 24, and Remark 16 of \cite{CarmonaEtAl19}.
In the following theorem, we summarize the mentioned results.
\begin{theorem}
\label{teo:defyL}
Let us consider system \eqref{ecu:linealL} and define the polynomial 
\begin{equation}\label{eq:polyL}
W_L(y)=D_Ly^2-a_LT_Ly+a_L^2.
\end{equation}
The forward Poincar\'{e} half-map $y_L$ is well defined if, and only if, $a_L\leqslant 0$ and $4D_L-T_L^2>0$, or $a_L>0$. In this case, its interval of definition $I_L:= [\lambda_L,\mu_L)\subset[0,+\infty)$ is non-empty and the following statements hold:
\begin{enumerate}[label = {(\alph*)}]
\item \label{muL} The right endpoint $\mu_L$ of the interval $I_L$ is the smallest strictly positive root of the polynomial $W_L$,  if it exists. Otherwise, $\mu_L=+\infty$.
\item \label{lambdaL} The left endpoint $\lambda_L$ of the interval $I_L$ is strictly positive if, and only if, $a_L<0$, $4D_L-T_L^2>0$, and $T_L<0$. In this case, $y_L(\lambda_L)=0$.
\item  \label{yL(IL)left} The left endpoint of the interval $y_L(I_L)$ is the largest strictly negative root of the polynomial $W_L$, if it exists. In the opposite case, this left endpoint is $-\infty$. 

\item\label{yL(IL)right} The right endpoint of the interval $y_L(I_L),$ that is $y_L(\lambda_L)$, is strictly negative if, and only if, $a_L<0$, $4D_L-T_L^2>0$, and $T_L>0$. In this case, $\lambda_L=0$.

\item\label{signWL} The polynomial $W_L$ satisfies $W_L(0)=a_L^2>0$ for $a_L\ne 0$  and $W_L(0)=0$ for $a_L=0$. Moreover,  $W_L(y)>0$ for $y \in \operatorname{ch}(I_L\cup y_L(I_L))\setminus\{0\}$, where $\operatorname{ch} (\cdot)$ denotes the convex hull of a set.
\item The forward Poincar\'{e} half-map $y_L$ is the unique function $y_L: I_L\subset [0,+\infty) \longrightarrow(-\infty,0]$
that satisfies 
\begin{equation}\label{integralF}
\operatorname{PV}\left\{\int_{y_L(y_0)}^{y_0}\dfrac{-y}{W_L(y)}dy\right\}= q_L(a_L,T_L,D_L)
\quad \mbox{for} \quad y_0\in I_L, 
\end{equation}
where
\begin{equation}
\label{ecu:qL}
q_L(a_L,T_L,D_L)=\left\{
\begin{array}{ccl}
0 & \mathrm{if} & a_L>0, \\
\frac{\pi T_L}{D_L\sqrt{4D_L-T_L^2}}
& \mathrm{if} & a_L=0, \\
\frac{2\pi T_L}{D_L\sqrt{4D_L-T_L^2}}
& \mathrm{if} & a_L<0.
\end{array}
\right.
\end{equation}
\item\label{rm:derivadaL} The graph of the forward Poincar\'{e} half-map, oriented according to increasing $y_0$, is the portion included in the fourth quadrant $\QQ$ of a particular orbit of the  
cubic vector field
\begin{equation}\label{dy1L}
X_L(y_0,y_1)=-\big(y_1W_L(y_0) ,y_0 W_L(y_1)\big).
\end{equation}
Equivalently,  the forward Poincar\'{e} half-map is a solution of the differential equation \begin{equation}
\label{eq:odeL}
y_1W_L(y_0)dy_1-y_0W_L(y_1)dy_0=0.
\end{equation}

\item\label{analyticL} The forward Poincar\'e half-map $y_L$ is analytic in $\Int(I_L)$. Moreover, $y_L$ is analytic in $I_L$ if, and only if, $\lambda_L=0$.
\end{enumerate}
 \end{theorem}

Now, because of the anticlockwise direction of the flow of \eqref{cf} again, the backward Poincar\'{e} half-map, $y_R$, is determined by the following linear differential system
\begin{equation}
\label{ecu:linealR}
\left\{\begin{array}{l}
\dot x= T_R x-y,\\
\dot y= D_R x-a_R,
\end{array}\right.
\end{equation}
which matches the right linear system of \eqref{cf}. Thus, its definition, its domain $I_R$, and its analyticity are obtained from Theorem \ref{teo:defyL} by means of the change of variables $(t,x)\mapsto(-t,-x)$ and taking $(a_L,D_L,T_L)=(-a_R,D_R,-T_R)$ in system \eqref{ecu:linealL}.
 \begin{theorem}
\label{teo:defyR}
Let us consider system \eqref{ecu:linealR} and define the polynomial 
\begin{equation}\label{eq:polyR}
W_R(y)=D_Ry^2-a_RT_Ry+a_R^2.
\end{equation}
The backward Poincar\'{e} half-map $y_R$ is well defined if, and only if, $a_R\geqslant 0$ and $4D_R-T_R^2>0$, or $a_R<0$. In this case, its interval of definition $I_R:=[\lambda_R,\mu_R)\subset[0,+\infty)$ is non-empty and the following statements hold:
\begin{enumerate}[label = {(\alph*)}]
\item\label{muR} The right endpoint $\mu_R$ of its definition interval $I_R$ is the smallest strictly positive root of the polynomial $W_R$, 
if it exists. Otherwise, $\mu_R=+\infty$.

\item\label{lambdaR} The left endpoint $\lambda_R$ of the interval $I_R$ is strictly positive if, and only if, $a_R>0$, $4D_R-T_R^2>0$, and $T_R>0$. In this case, $y_R(\lambda_R)=0$.

\item\label{yR(IR)left} The left endpoint of the interval $y_R(I_R)$ is the largest strictly negative root of the polynomial $W_R$, if it exists. In the opposite case, this left endpoint is $-\infty$.

\item\label{yR(IR)right} The right endpoint of the interval $y_R(I_R),$ that is $y_R(\lambda_R)$, is strictly negative if, and only if,  $a_R>0$, $4D_R-T_R^2>0$, and $T_R<0$. In this case, $\lambda_R=0$.

\item\label{signWR} The polynomial $W_R$ satisfies $W_R(0)=a_R^2>0$ for $a_R\ne 0$  and $W_R(0)=0$ for $a_R=0$. Moreover, $W_R(y)>0$ for $y \in \operatorname{ch}(I_R\cup y_R(I_R))\setminus\{0\}$.
\item The backward Poincar\'{e} half-map $y_R$ is the unique function $y_R: I_R\subset [0,+\infty) \longrightarrow(-\infty,0]$
that satisfies 
\begin{equation}\label{integralB}
\operatorname{PV}\left\{\int_{y_R(y_0)}^{y_0}\dfrac{-y}{W_R(y)}dy\right\}= q_R(a_R,T_R,D_R)
\quad \mbox{for} \quad y_0\in I_R,
\end{equation}
where
\begin{equation}
\label{ecu:qR}
q_R(a_R,T_R,D_R)=\left\{
\begin{array}{ccl}
0 & \mathrm{if} & a_R<0, \\
-\frac{\pi T_R}{D_R\sqrt{4D_R-T_R^2}}
& \mathrm{if} & a_R=0, \\
-\frac{2\pi T_R}{D_R\sqrt{4D_R-T_R^2}}
& \mathrm{if} & a_R>0.
\end{array}
\right.
\end{equation}
\item\label{rm:derivadaR} The graph of the backward Poincar\'{e} half-map, oriented according to increasing $y_0$, is the portion included in the fourth quadrant $\QQ$ of a particular orbit of the cubic vector field
\begin{equation}\label{dy1R}
X_R(y_0,y_1)=-\big(y_1W_R(y_0) ,y_0 W_R(y_1)\big).
\end{equation}
Equivalently, the backward Poincar\'{e} half-map is a solution of the differential equation 
\begin{equation}
\label{eq:odeR}
y_1W_R(y_0)dy_1-y_0W_R(y_1)dy_0=0.
\end{equation}

\item\label{analyticR} The backward Poincar\'e half-map $y_R$ is analytic in $\Int(I_R)$. Moreover, $y_R$ is analytic in $I_R$ if, and only if, $\lambda_R=0$.
\end{enumerate}
\end{theorem}

In the next remark, we provide a characterization to easily distinguish, under the generic condition $a_L T_La_RT_R\neq0$, when the Poincar\'e half-maps transform $0$ to $0$. This will be used in the proof of propositions \ref{prop:estabmonodromico} and 
\ref{prop:hyperlimitcycle}.

\begin{remark}
\label{rem:0en0}
Under the condition $a_LT_L\neq0$, by assuming the existence of $y_L$, it follows from the statements \ref{lambdaL} and \ref{yL(IL)right} of Theorem \ref{teo:defyL} that $0\in I_L$ and  $y_L(0)=0$ if, and only if, $a_L>0$.

Analogously, under the condition $a_RT_R\neq0$, by assuming the existence of $y_R$, it follows from the statements \ref{lambdaR} and \ref{yR(IR)right} of Theorem \ref{teo:defyR} that $0\in I_R$ and  $y_R(0)=0$ if, and only if, $a_R<0$.
\end{remark}

Now, we establish some fundamental properties of the Poincar\'{e} half-maps which will be useful for the proof of the main result.  The proofs of these properties  can be found in \cite{Caretalpre21}. 

From \eqref{eq:odeL} and \eqref{eq:odeR}, it is straightforward to obtain explicit expressions for the first and second derivatives of $y_L$ and $y_R$. 
\begin{proposition}
\label{prop:derivadasprimeraysegundayLyR}
The first and second derivatives of the Poincar\'{e} half-maps $y_L$ and $y_R$ are given by 
\begin{equation*}
\label{eq:derivadaprimerayL}
y_L'(y_0)=\frac{y_0W_L(y_L(y_0))}{y_L(y_0)W_L(y_0)}    \quad \mbox{for} \quad y_0\in \operatorname{int}(I_L),
\end{equation*}
\begin{equation*}
\label{eq:derivadaprimerayR}
y_R'(y_0)=\frac{y_0W_R(y_R(y_0))}{y_R(y_0)W_R(y_0)} \quad \mbox{for} \quad y_0\in \operatorname{int}(I_R),
\end{equation*}
\begin{equation}
\label{eq:derivadasegundayL}
y''_L(y_0)=-\frac{a_L^2 \left(y_0^2-\left(y_L(y_0)\right)^2\right)W_L(y_L(y_0))}{\left(y_L(y_0)\right)^3\left(W_L(y_0)\right)^2} \quad \mbox{for} \quad y_0\in \operatorname{int}(I_L),
\end{equation}
and 
\begin{equation}
\label{eq:derivadasegundayR}
y''_R(y_0)=-\frac{a_R^2 \left(y_0^2-\left(y_R(y_0)\right)^2\right)W_R(y_R(y_0))}{\left(y_R(y_0)\right)^3\left(W_R(y_0)\right)^2} \quad \mbox{for} \quad y_0\in \operatorname{int}(I_R).
\end{equation}
\end{proposition}

Now, we show the first coefficients of the Taylor expansions of the Poincar\'{e} half-maps at the origin (see Proposition \ref{prop:sequ}) and at infinity (see Proposition \ref{prop:sinf}), under suitable conditions. These are the essentials for our study. More coefficients and other details can be seen in \cite{Caretalpre21}.

\begin{proposition}\label{prop:sequ}
Assume that $a_L\neq0$ (resp. $a_R\ne 0$) and $0\in I_L$ (resp. $0\in I_R$). If $y_L(0)=0$ (resp. $y_R(0)=0$), then the Taylor expansion of $y_L$ (resp. $y_R$) around the origin writes as
\begin{equation*}\label{taylor-yLR}
y_L(y_0)=-y_0-\dfrac{2T_L}{3a_L}y_0^2+\mathcal{O} \left(y_0^3\right)  \quad \mbox{(resp. $y_R(y_0)=-y_0-\dfrac{2T_R}{3a_R}y_0^2+\mathcal{O}  \left(y_0^3\right)$).}
\end{equation*}
\end{proposition}
Notice that if $4D_L-T_L^2>0$  and $4D_R-T_R^2>0$, then, from the statements \ref{muL} and \ref{yL(IL)left} of Theorem \ref{teo:defyL} and the statements \ref{muR} and \ref{yR(IR)left} of Theorem \ref{teo:defyR}, the intervals $I_L$ and $I_R$ are unbouded and $y_L(y_0)$ and $y_R(y_0)$ tend to $-\infty$ as $y_0\to +\infty$.

\begin{proposition}\label{prop:sinf} The following statements hold.
\begin{enumerate}
\item If $4D_L-T_L^2>0$, then the forward Poincar\'e half-map $y_L$ satisfies
\[
\lim_{y_0\to +\infty}\frac{y_L(y_0)}{y_0}=-\exp\left(\frac{\pi T_L}{\sqrt{4D_L-T_L^2}}\right).
\]
\item If $4D_R-T_R^2>0$, then the backward Poincar\'e half-map $y_R$ satisfies
\[
 \lim_{y_0\to +\infty}\frac{y_0}{y_R(y_0)}=-\exp\left(\frac{\pi T_R}{\sqrt{4D_R-T_R^2}}\right).
\]
\end{enumerate}
\end{proposition}

We conclude this section with two results that establish the relative position between the graph of the  Poincar\'{e} half-maps and the bisector of the fourth quadrant, that is, the half straight line $y_1=-y_0$, $y_0\geq0$. In Figure \ref{fig:casesTL}, we show the possible relative positions between the graph of the Poincar\'{e} map $y_L$ and the bisector of the fourth quadrant by varying the trace $T_L$ for three illustrative cases. A similar figure could be given for the map $y_R$ by varying the trace $T_R$.

\begin{proposition}
\label{rm:signoy0+y1}
The following statements hold.
\begin{enumerate}[label = {(\alph*)}]
\item The forward Poincar\'{e} half-map $y_L$ satisfies the relationship
\[
\sgn\left(y_0+y_L(y_0) \right)=-\sgn(T_L) \quad \mbox{for} \quad y_0\in  I_L\setminus\{0\}.
\]
In addition,
when $0\in I_L$ and $y_L(0)\neq0$ or when $T_L=0$, then the relationship above also holds for $y_0=0$.
\item The backward Poincar\'{e} half-map $y_R$ satisfies the relationship
\[
\sgn\left(y_0+y_R(y_0) \right)=\sgn(T_R) \quad \mbox{for} \quad y_0\in  I_R\setminus\{0\}.
\]
In addition, 
when $0\in I_R$ and $y_R(0)\neq0$ or when $T_R=0$, then the relationship above also holds for $y_0=0$.
\end{enumerate}
\end{proposition} 

\begin{corollary}\label{bi2quad}
The following statements are true.
\begin{enumerate}[label = {(\alph*)}]
\item If $T_L=0$ (resp. $T_R=0$), then the graph of the Poincar\'{e} half-map $y_L$ (resp. $y_R$)
is included in the bisector of the fourth quadrant. 
\item If $T_L\ne 0$ (resp. $T_R\ne 0$), then the graph of the Poincar\'{e} half-map $y_L$ (resp. $y_R$)
does not intersect  the bisector of the fourth quadrant except perhaps at the origin.
\end{enumerate}
\end{corollary}

\begin{figure}[H]
\begin{minipage}{0.31\linewidth}
\centering
\includegraphics[width=4.5cm]{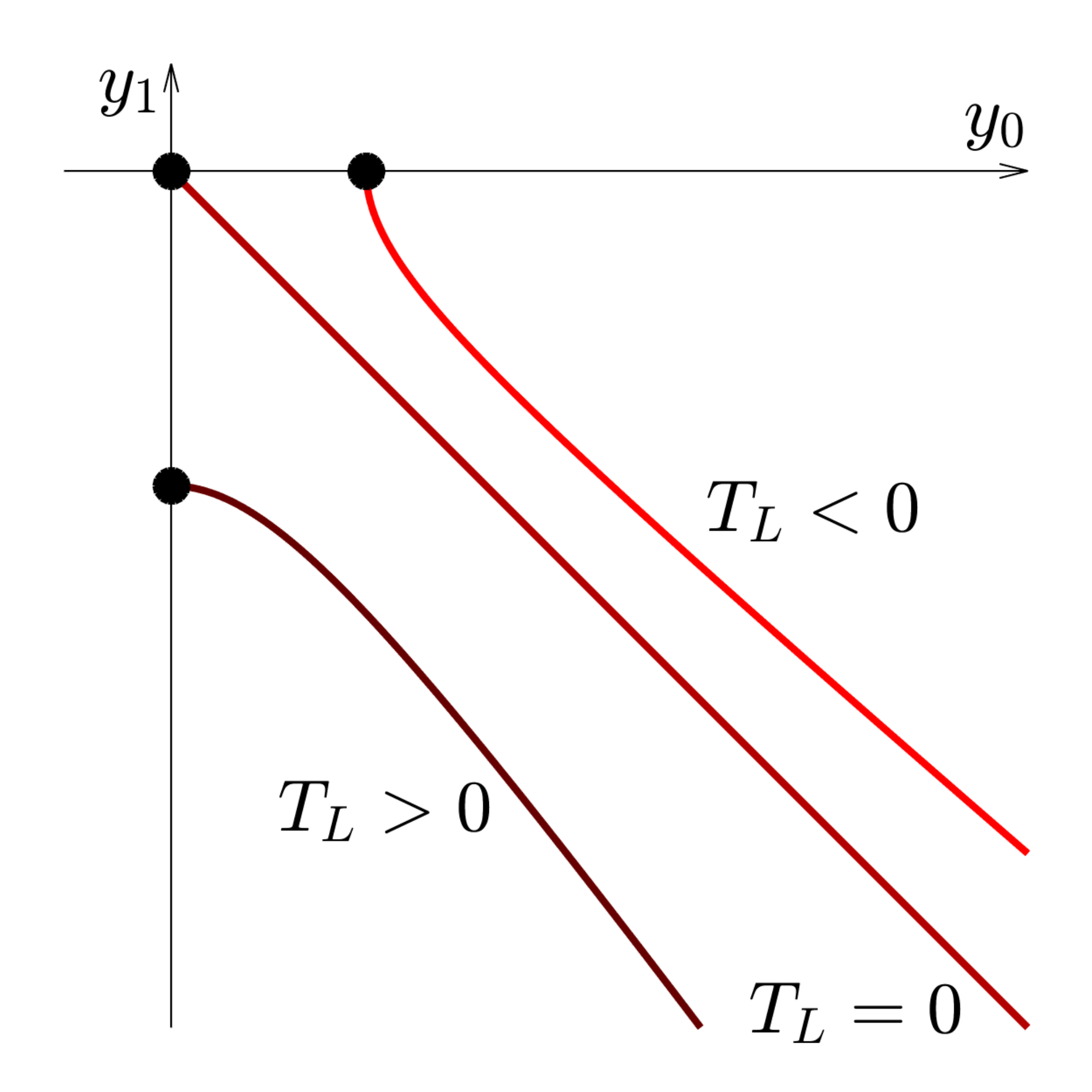}
\end{minipage}
\begin{minipage}{0.31\linewidth}
\centering
\includegraphics[width=4.5cm]{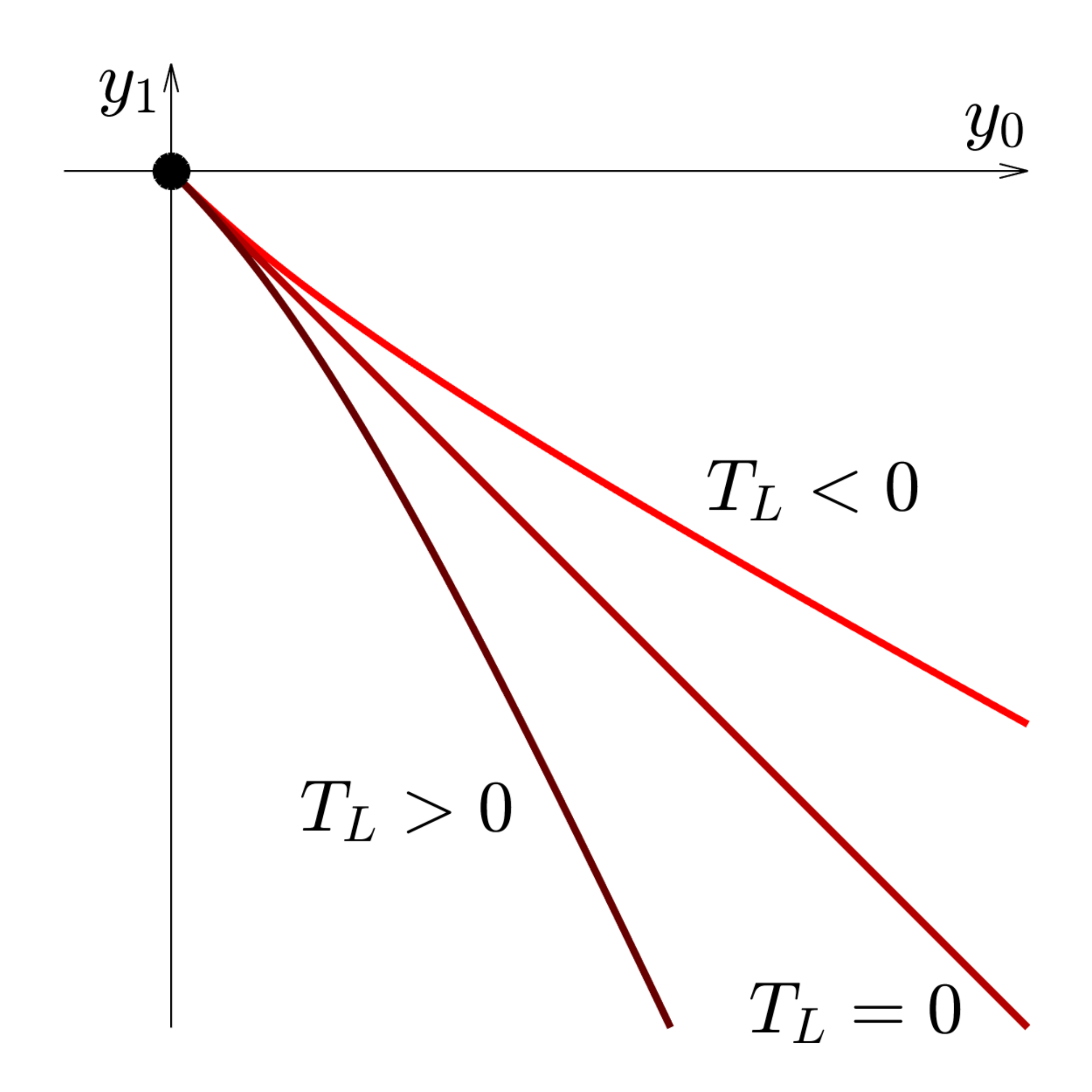}
\end{minipage}
\begin{minipage}{0.31\linewidth}
\centering
\includegraphics[width=4.5cm]{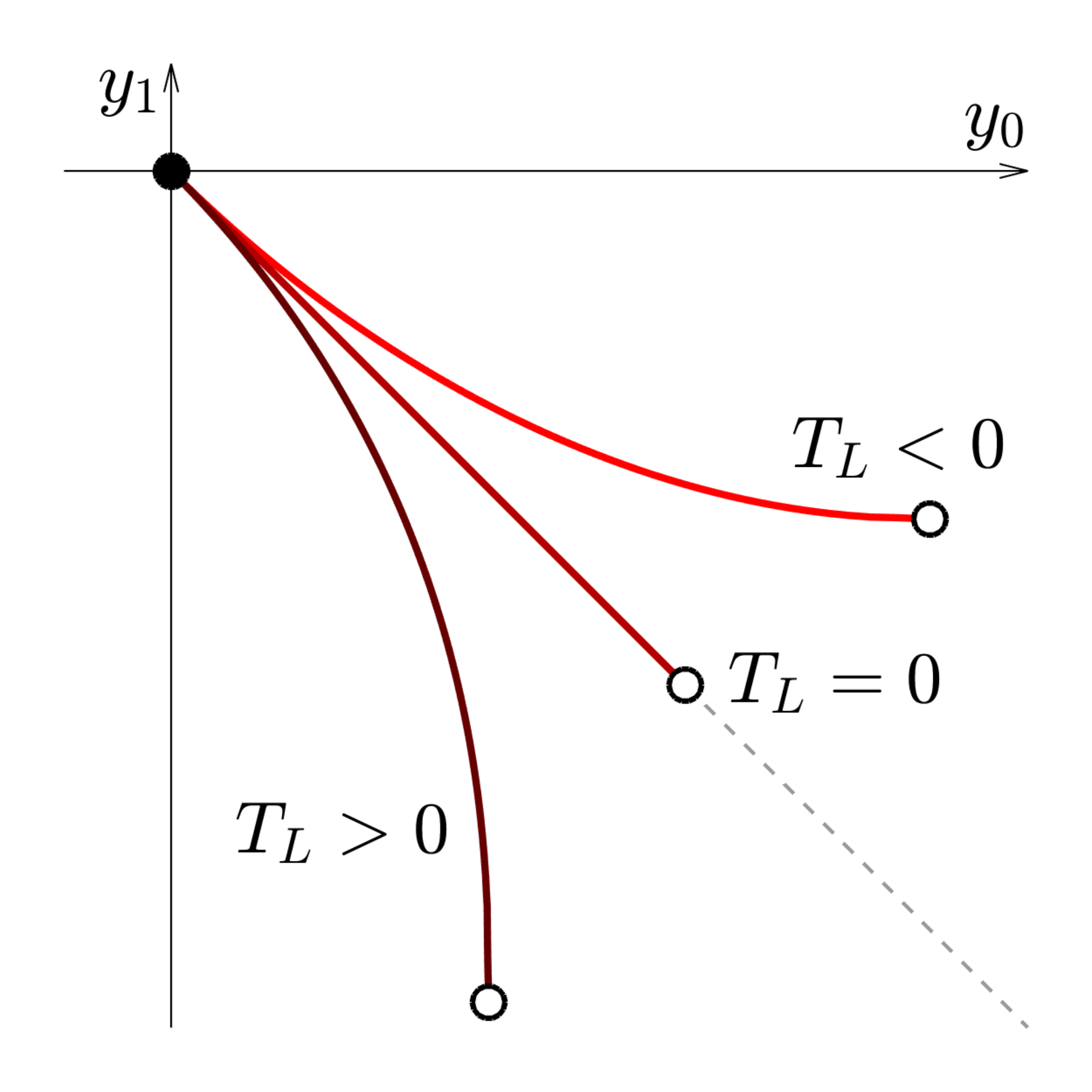}
\end{minipage}
\caption{Three illustrative cases of relative positions between the graph of the Poincar\'{e} half-map $y_L$ and the bisector of the fourth quadrant in terms of the trace $T_L$. The left figure corresponds to a real focus (that is, negative $x$-coordinate) with $D_L=1,$ $a_L=-1/2$, and $T_L\in\{-1/10,\,0,\,1/5\}$; the middle one corresponds to a virtual focus (that is, positive $x$-coordinate) with $D_L=1,$ $a_L=1/2$, and $T_L\in\{-2/5,\,0,\,1/2\}$; and the right one corresponds to a  real saddle with $D_L=-1,$ $a_L=6/5$, and $T_L\in\{-4/5,\,0,\,1\}$. The black circles represent the starting point of the graphs of $y_L$ while, in the right figure, the white circles represent the ending point of the graph of  $y_L$,  whose  coordinates are  the roots of the polynomial 
$W_L(y)=D_Ly^2-a_LT_Ly+a_L^2$ provided in  \eqref{eq:polyL} (which also correspond to the $y$-coordinates of the intersection points between the invariant straight lines of the real saddle with the switching line $\Sigma$). Moreover, in the right figure, the dashed line corresponds to the the part of the bisector of the fourth quadrant which is not hidden by the graph of the Poincar\'{e} half-map.}\label{fig:casesTL}
\end{figure}

\section{Displacement function: fundamental properties}
\label{sec:displfunct}

In this section, as it is usual for the analysis of limit cycles, a displacement function whose zeroes correspond to the periodic orbits of sytem \eqref{cf} will be defined by means of the difference of the Poincar\'e half-maps. In addition, its lower order derivatives will be computed to later provide us with the hyperbolicity and stability of the periodic orbits. 

Notice that any limit cycle of the differential system \eqref{cf} is anticlockwise oriented and transversally crosses the switching line $\Sigma$ twice. Indeed, a limit cycle cannot be entirely contained in the closure of a zone of linearity and cannot pass through the origin which is the unique tangency point of the flow of the differential system  \eqref{cf} in the switching line. 

Moreover, from now on till the end of this paper, we assume $I_L\cap I_R\neq\emptyset$ because this is a necessary condition for the existence of crossing periodic orbits of the differential system \eqref{cf}. Notice that this trivially implies the existence of the Poincar\'e half-maps and includes the conditions $a_L^2+D_L^2\ne 0$ and $a_R^2+D_R^2\ne 0$.

The displacement function $\delta$ is now defined in the interval $I:=I_L\cap I_R=[\lambda_0,\mu_0)\neq\emptyset$ as
\begin{equation}\label{displacement}
\begin{array}{cccl}
\delta : & I & \longrightarrow & \mathbb{R} \\
  & y_0 & \longmapsto & \delta(y_0)= y_R(y_0)-y_L(y_0),
\end{array}
\end{equation}
where $\lambda_0=\max\{\lambda_L,\lambda_R\}$ and $\mu_0=\min\{\mu_L,\mu_R\}$. In Figure \ref{fig:displacement}, we illustrate the displacement function for a scenario with a real focus and a real saddle.

\begin{figure}[H]
    \begin{center}
     \includegraphics[width=6cm]{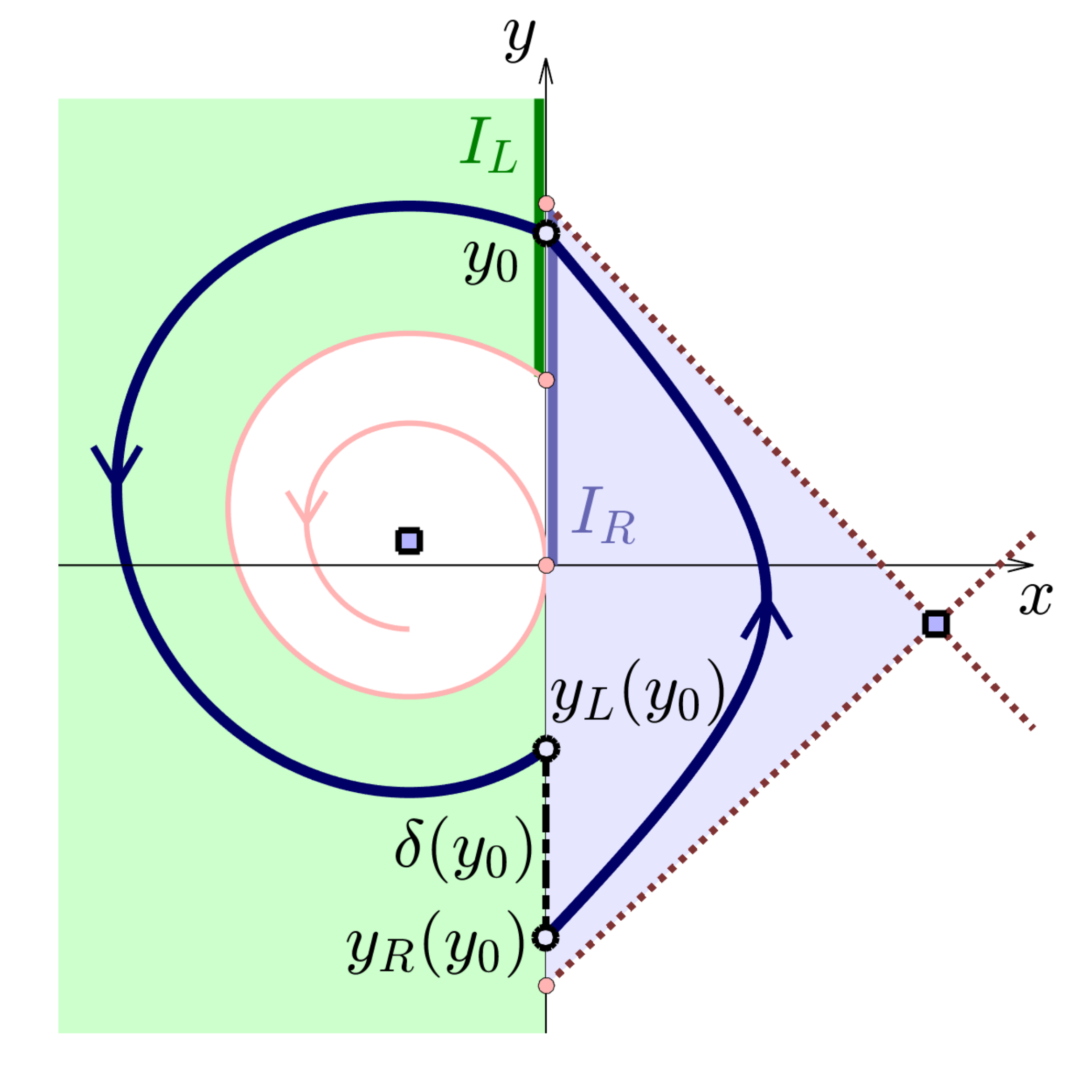}
    \end{center}
     \caption{An example to illustrate the displacement function. It corresponds to a piecewise linear system composed by a (left) system with a real focus and a (right) system with a real saddle.}\label{fig:displacement}
\end{figure}

Notice that a crossing periodic orbit exists and passes through $(0,y_0^*),$ with $y_0^*\in \operatorname{int}(I)$, if, and only if,  $y_0^*$ is a zero of displacement function $\delta$. Such a periodic orbit is a hyperbolic limit cycle provided that $\delta'(y_0^*)\neq0$. In this case, it is attracting (resp. repelling) provided that $\delta'(y_0^*)<0$ (resp. $\delta'(y_0^*)>0$). 
In the next result, a relevant expression for the sign of this derivative will be deduced from Proposition \ref{prop:derivadasprimeraysegundayLyR}. 

\begin{proposition}
\label{prop:derivdisplafunct} Suppose that $y_0^*\in \operatorname{int}(I)$ satisfies $\delta(y_0^*)=0$. Denote $y_1^*=y_R(y_0^*)=y_L(y_0^*)<0$ and define
\begin{equation}
\label{eq:c0c1c2}
\begin{array}{l}
c_0:=a_Ra_L\left(a_RT_L-a_LT_R\right),\\
c_1:=a_RT_RD_L-a_LT_LD_R,\\
c_2:=a_L^2 D_R-a_R^2 D_L.\\
\end{array}
\end{equation}
 Then, the following statements hold:
 \begin{enumerate}[label = {(\alph*)}]
\item The derivative of the displacement function $\delta$ defined in \eqref{displacement} verifies
\begin{equation}
\label{eq:signodeltaprima}
\sgn(\delta'(y_0^*))=\sgn(F(y_0^*,y_1^*)),
\end{equation}
being 
\begin{equation}\label{CeF}
F(y_0,y_1)=c_0+c_1 y_0 y_1+c_2(y_0+y_1).
\end{equation}
\item Moreover, if $\delta'(y_0^*)=0$, then the second derivative of $\delta$ verifies
\begin{equation}
\label{eq:signodeltasegunda}
\sgn(\delta''(y_0^*))=\sgn\left(T_L \left( c_2 y_0^*+c_0\right) \right)=-\sgn\left(T_R \left( c_2 y_1^*+c_0\right) \right).
\end{equation}
\end{enumerate}
\end{proposition}
\begin{proof}
Let us fix $y_0^*\in \operatorname{int}(I)$ with $y_1^*=y_R(y_0^*)=y_L(y_0^*)
$. From Proposition \ref{prop:derivadasprimeraysegundayLyR}, it is straightforward to get \begin{equation*}
\label{eq:derivadadelta}
\delta'(y_0^*)=y'_R(y_0^*)-y'_L(y_0^*)=C(y_0^*,y_1^*)F(y_0^*,y_1^*),
\end{equation*}
where  $C$ and $F$ are given by
\[
C(y_0,y_1)=\dfrac{-y_0(y_0-y_1)}{y_1W_R(y_0)W_L(y_0)}
\]
and
\begin{equation}\label{Fe1}
F(y_0,y_1)=\dfrac{W_L(y_1)W_R(y_0)-W_L(y_0)W_R(y_1)}{y_0-y_1}.
\end{equation}

By substituting expressions \eqref{eq:polyL} and \eqref{eq:polyR} of polynomials $W_L$ and $W_R$ in expression \eqref{Fe1}, we see that, for $y_0\ne y_1$,
\[
F(y_0,y_1)=c_0+c_1 y_0 y_1+c_2(y_0+y_1),
\]
where the coefficients $c_0,c_1,$  and $c_2$ are given in \eqref{eq:c0c1c2}. 

Since $y_0^*>0, y_1^*< 0,$ $W_R(y_0^*)>0,$ and  $W_L(y_0^*)>0$ (see the statements \ref{signWL} from theorems \ref{teo:defyL} and \ref{teo:defyR}), then $C(y_0^*,y_1^*)>0$ and item (a) is proven.

Now, from equations \eqref{eq:derivadasegundayL} and \eqref{eq:derivadasegundayR} of Proposition \ref{prop:derivadasprimeraysegundayLyR},
it follows that
\begin{equation}
\label{eq:derivadasegundadelta}
\delta''(y_0^*)=y''_R(y_0^*)-y''_L(y_0^*)=\frac{\left(y_0^*\right)^2-\left(y_1^*\right)^2}{\left(y_1^*\right)^3}
\left[
\frac{a_L^2 W_L(y_1^*)}{W_L(y_0^*)^2}-\frac{a_R^2 W_R(y_1^*)}{W_R(y_0^*)^2}
\right].
\end{equation}
Taking into account equalities \eqref{eq:signodeltaprima} and \eqref{Fe1}, it is clear that
$\delta'(y_0^*)=0$ if, and only if, $W_L(y_1^*)W_R(y_0^*)=W_L(y_0^*)W_R(y_1^*)$. Therefore, the assumption $\delta'(y_0^*)=0$ allows to write the second derivative \eqref{eq:derivadasegundadelta} as
\begin{equation*}
\label{eq:derivadasegundadeltaconderivadacero}
\begin{array}{rcl}
\delta''(y_0^*)&=&\displaystyle \frac{\left(y_0^*\right)^2-\left(y_1^*\right)^2}{\left(y_1^*\right)^3}
\;\frac{W_R(y_1^*)W_L(y_0^*)}{\left(W_L(y_0^*)W_R(y_0^*)\right)^2} 
\left[
a_L^2 W_R(y_0^*)-a_R^2 W_L(y_0^*)
\right]\\
\noalign{\medskip}
&=&
\displaystyle
\left(y_0^*+y_1^*\right) 
\frac{y_0^*-y_1^*}{\left(y_1^*\right)^3}\;
\frac{W_R(y_1^*)}{W_L(y_0^*) \left(W_R(y_0^*)\right)^2} 
\;
y_0^*
\left[
c_2 y_0^*+c_0
\right].
\end{array}
\end{equation*}
Since $y_0^*>0, y_1^*< 0,$ $W_R(y_1^*)>0,$ $W_L(y_0^*)>0$, from Proposition \ref{rm:signoy0+y1} we obtain
$$
\sgn{\left(\delta''(y_0^*)\right)}=-\sgn{\left(\left(y_0^*+y_1^*\right) \left(
c_2 y_0^*+c_0\right)\right)}=\sgn\left(T_L \left( c_2 y_0^*+c_0\right) \right).
$$
Analogous computations allow to prove the second equality of equation \eqref{eq:signodeltasegunda}. Hence, the proof of item (b) is finished.
\end{proof}

\begin{remark}\label{rem:multiplier}
Notice that the coefficient $c_0$ given in \eqref{eq:c0c1c2} is the product of three factors. The factors $a_L$ and $a_R$ are the nonhomogeneous terms of the differential system \eqref{cf} and
\begin{equation} \label{eq:xi}
\xi:=a_RT_L-a_LT_R
\end{equation}
has already appeared in Theorem \ref{main} and, as will be proven, determines the stability of the unique limit cycle, if any.
\end{remark}

\begin{remark}\label{rem:hype}
Let us give some comments about the zero set $\gamma=F^{-1}(\{0\})$, for the function $F$ given in \eqref{CeF}. When $c_1 (c_2^2-c_1 c_0)=0$, the set $\gamma$ is either the empty set, the whole plane, a straight line, or a pair of secant straight lines. We shall see that these non-generic cases will not be relevant in our analysis. Therefore, in this Remark, we assume that $c_1 (c_2^2-c_1 c_0)\neq0$ and, consequently, the set $\gamma$
is a non-degenerate hyperbola.

The asymptotes of $\gamma$ are given by $y_0=-c_2/c_1$ and $y_1=-c_2/c_1$.
 On the one hand, except for the case $c_2=0$ (when the asymptotes coincide with the coordinate axes), the two intersection points between the hyperbola and the coordinate axes are $(-c_0/c_2,0)$ and $(0,-c_0/c_2)$. On the other hand, the center of $\gamma$, (that is, the intersection point between its asymptotes) is the point $(-c_2/c_1, -c_2/c_1)$. Note that a center like that must lay either on the origin or in the first or third quadrant. Thus, at most one branch of the hyperbola may intersect the fourth quadrant $\QQ$, which would then be divided into two connected components.

Regarding the intersection between $\gamma$ and the bisector of the fourth quadrant, notice that it only occurs when $c_0c_1\geqslant0$. In this case, the point of intersection is
\begin{equation}
\label{eq:puntito}
(\widebar{y}_0,\widebar{y}_1)=\left(\sqrt{\dfrac{c_0}{c_1}},-\sqrt{\dfrac{c_0}{c_1}}\right).
\end{equation}
Furthermore, it can be seen that if $c_0\neq0$ then the intersection is transversal, and if $c_0=0$ then the point of intersection is the origin and the intersection is non-transversal.

Finally, it is easy to see that the function $y_1=\phi(y_0),$ which describes the hyperbola $\gamma$ as a graph, is increasing when $c_0c_1-c_2^2>0$ and decreasing when $c_0c_1-c_2^2<0$.
\end{remark}

\section{Stability of monodromic singularities and the infinity}\label{sec:stability}

This section will be devoted to presenting some results about stability for monodromic singularities. 

In broad terms, the concept of monodromy is related to the rotation of the flow of the differential system. In order to precisely establish this concept for a point of the phase plane of the differential system \eqref{cf}, it is convenient to distinguish if it is located in the switching line $\Sigma$ or not:

\begin{itemize}[leftmargin=*]
\item A point of the differential system \eqref{cf} outside the switching line $\Sigma$ is said to be a {\it monodromic singularity} when it is a focus or a center.

\item Due to the crossing behavior of the flow of the differential system \eqref{cf} on $\Sigma$, the origin is the unique point of $\Sigma$ around which rotation is possible and, moreover, this rotation is a consequence of the half-rotations for the half-planes $\left\{(x,y)\in\mathbb{R}^2:x\leqslant 0\right\}$ and $\left\{(x,y)\in\mathbb{R}^2:x\geqslant 0\right\}$. A glimpse of the vector field of the differential system \eqref{cf}  allows to assert that the half-rotation around the origin in the half-plane $\left\{(x,y)\in\mathbb{R}^2:x\leqslant 0\right\}$ only occurs when $a_L=0$ and $4D_L-T_L^2>0$ or when $a_L>0$. Analogously,  the half-rotation around the origin in the half-plane $\left\{(x,y)\in\mathbb{R}^2:x\geqslant 0\right\}$ only happens when $a_R=0$ and $4D_R-T_R^2>0$ or when $a_R<0$.  Therefore, the origin is said to be a {\it monodromic singularity} if, and only if, one of the four feasible combinations above holds. 
\end{itemize}

The next result provides the stability of the origin when it is a monodromic singularity of \eqref{cf} in the case $a_L>0$ and $a_R< 0$.

\begin{proposition}
\label{prop:ep}
Let  $\xi$ be defined as in expression \eqref{eq:xi}. 
Consider the differential system \eqref{cf} with $a_L>0$ and $a_R< 0$. Then, 
the Poincar\'{e} half-maps satisfy $y_R(0)=y_L(0)=0$ and 
the origin is an attracting (resp. repelling) monodromic singularity provided that $\xi > 0$ (resp. $\xi < 0$).
\end{proposition}
\begin{proof}
When $a_L>0$, from expression \eqref{integralF}, it follows that $y_L(0)=0$. In the same manner, if $a_R<0$, then $y_R(0)=0$. Since $a_L>0$ and $a_R<0$, the origin is a monodromic singularity. 
Its stability is obtained by computing the power series of the displacement function \eqref{displacement}  around $y_0 = 0$. 
Accordingly, from Proposition \ref{prop:sequ}, we get
\begin{equation}
\label{eq:deltarrollo}
\delta(y_0)=\left(\dfrac{2T_{L}}{3a_{L}}-\dfrac{2T_{R}}{3a_{R}}\right)y_0^2+\CO(y_0^3)
=\dfrac{2 \xi}{3 a_L a_R}y_0^2+\CO(y_0^3).
\end{equation}

Therefore, the monodromic singularity is attracting (resp. repelling) provided that $\xi>0$ (resp. $\xi<0$).
\end{proof}

 \begin{remark}
Taking expression \eqref{eq:deltarrollo} into account we see that, under the hypotheses of Proposition \ref{prop:ep}, the first Lyapunov coefficient of the monodromic singularity of \eqref{cf} at the origin (see, for instance, \cite{gassulcoll,NS21}) is given by the expression $2 \xi/(3 a_L a_R).$ Thus, under suitable assumptions, the parameter $\xi$ can be used, among other things, to generate a limit cycle through a Hopf-like bifurcation.
\end{remark}

The next result states that, in the case of a unique monodromic singularity and $T_LT_R<0$, the parameter $\xi$  also provides its stability, even if it is outside $\Sigma$.

\begin{proposition}
\label{prop:estabmonodromico} Assume that both (forward and backward) Poincar\'e half-maps exist.
Let $c_0$ and $\xi$ be as given in \eqref{eq:c0c1c2}  and \eqref{eq:xi}, respectively.
Consider the differential system \eqref{cf} with $T_LT_R<0$ and $c_0\neq 0$.
If the differential system \eqref{cf} admits a unique monodromic singularity, then it is attracting (resp. repelling) provided that $\xi>0$ (resp. $\xi < 0$). Moreover, if the displacement function $\delta$ can be defined, then $\sgn(\delta(y_0))=-\sgn(\xi)$ for $y_0$ close enough to the left endpoint $\lambda_0$ of the interval of definition $I$ of the displacement function $\delta$.
\end{proposition}
\begin{proof}
If the differential system \eqref{cf} admits a unique monodromic singularity and it belongs to section $\Sigma$, then the singularity is the origin. This, together with the hypothesis $c_0\neq0$ means that $a_L>0$ and $a_R<0$. The conclusion regarding the stability of the monodromic singularity follows from Proposition \ref{prop:ep}. The conclusion regarding the sign of $\delta$ close to $\lambda_0=0$ follows from expression \eqref{eq:deltarrollo}.

If the unique monodromic singularity belongs to the half-plane $\{ (x,y)\in\mathbb{R}^2: x<0\}$, then it is a center or a focus, which implies that $4D_L-T_L^2>0$ and so $D_L>0$ and $a_L<0$. 
On the one hand, $c_0\ne 0$ implies $a_R\ne 0$. On the other hand, since the existence of both Poincar\'e half-maps is assumed,
if $a_R>0$ then the characterization \eqref{integralB} would imply the existence of another monodromic singularity in the half-plane $\{ (x,y)\in\mathbb{R}^2: x>0\}$, which contradicts the hypotheses. Therefore, one has $a_R<0$ and,  from Remark \ref{rem:0en0}, $\lambda_R=0$ and $y_R(0)=0$. Moreover, since the condition $T_LT_R<0$ holds then
$\sgn(\xi)=-\sgn(T_L)\neq0$. As it is obvious from the linearity, the stability of the singularity is given by $\sgn(T_L)$ and so it is an attracting (resp. repelling) focus provided that $\xi> 0$ (resp. $\xi< 0$). 

In order to see that $\sgn(\delta(y_0))=-\sgn(\xi)$ for $y_0$ close enough to $\lambda_0=\lambda_L$ we distinguish the following two possible cases, namely, $\lambda_L=0$ and $\lambda_L>0$. 

If $\lambda_L=0$, taking into account that $T_L\neq0$, $a_L<0,$ and $4D_L-T_L^2>0$, Theorem \ref{teo:defyL}\ref{lambdaL} implies that $T_L>0$. Moreover, Theorem \ref{teo:defyL}\ref{yL(IL)right} provides $y_L(0)<0$. Thus, $\delta(\lambda_0)=y_R(0)-y_L(0)=-y_L(0)>0$ and, therefore, $\sgn(\delta(\lambda_0))=\sgn(T_L)=-\sgn(\xi)$. The conclusion in this case follows from the continuity of $\delta$.

If $\lambda_L>0$, Theorem \ref{teo:defyL}\ref{lambdaL} implies that $T_L<0$ and $y_L(\lambda_L)=0$. Thus, $\delta(\lambda_0)=y_R(\lambda_L)-y_L(\lambda_L)=y_R(\lambda_R)<0$ and, therefore, $\sgn(\delta(\lambda_0))=\sgn(T_L)=-\sgn(\xi)$. The conclusion in this case follows again from the continuity of $\delta$.

An analogous reasoning can be done if the  monodromic singularity belongs to the half-plane $\{ (x,y)\in\mathbb{R}^2: x>0\}$ and this concludes the proof.
\end{proof}

\begin{remark}
\label{rem:estabmonodromico}
If a monodromic singularity belongs to section $\Sigma$ then, from its definition, the existence of the Poincar\'e half-maps is ensured. Therefore, the assumption of their existence in Proposition \ref{prop:estabmonodromico} is not necessary in this case.

However, this assumption is needed if the monodromic singularity is not located in the switching line $\Sigma$. Indeed, consider  for instance the case $a_L<0$, $4D_L-T_L^2>0$, $T_L\neq0$, $a_R>0$, $D_R<0$ and $T_R=a_R T_L/a_L$. On the one hand, this case provides $\xi=0$ and $T_LT_R<0$. On the other hand, it corresponds
to a differential system, for which the backward Poincar\'e half-map does not exist  and with a focus equilibrium located in the half-plane $\{(x,y)\in\mathbb{R}:x<0\}$. Obviously, $\xi$ does not determine the stability of the focus, which is given by the sign of $T_L$ (see Figure \ref{fig:remark5}).

Furthermore, although we are assuming the existence of the Poincar\'e half-maps, we stress that the existence of the displacement function $\delta$ is not necessary for the conclusion of Proposition \ref{prop:estabmonodromico}.
\end{remark}

\begin{figure}[H]
\begin{minipage}{0.49\linewidth}
\centering
\includegraphics[width=4.2cm]{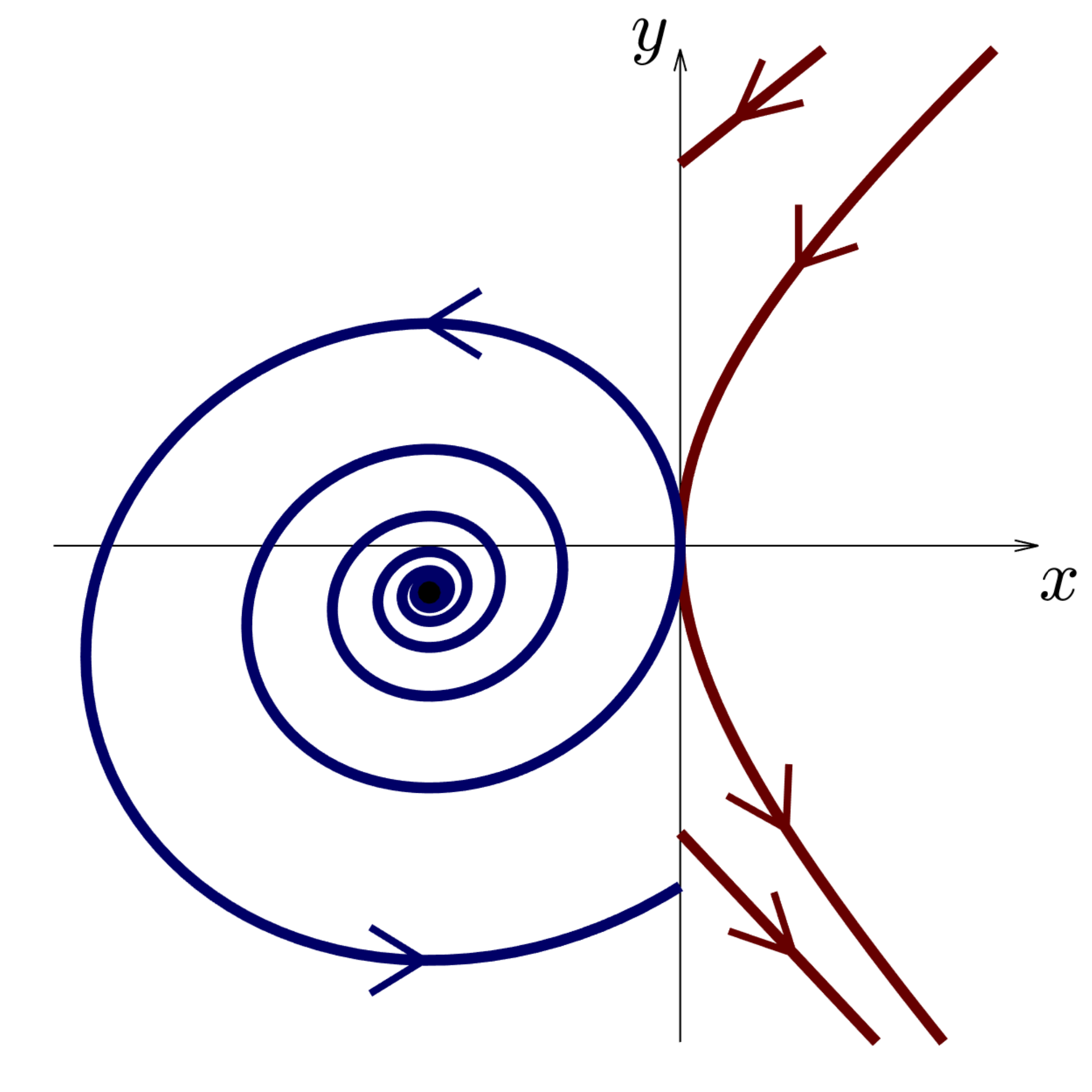}\\
(a)
\end{minipage}
\hfill
\begin{minipage}{0.49\linewidth}
\centering
\includegraphics[width=4.2cm]{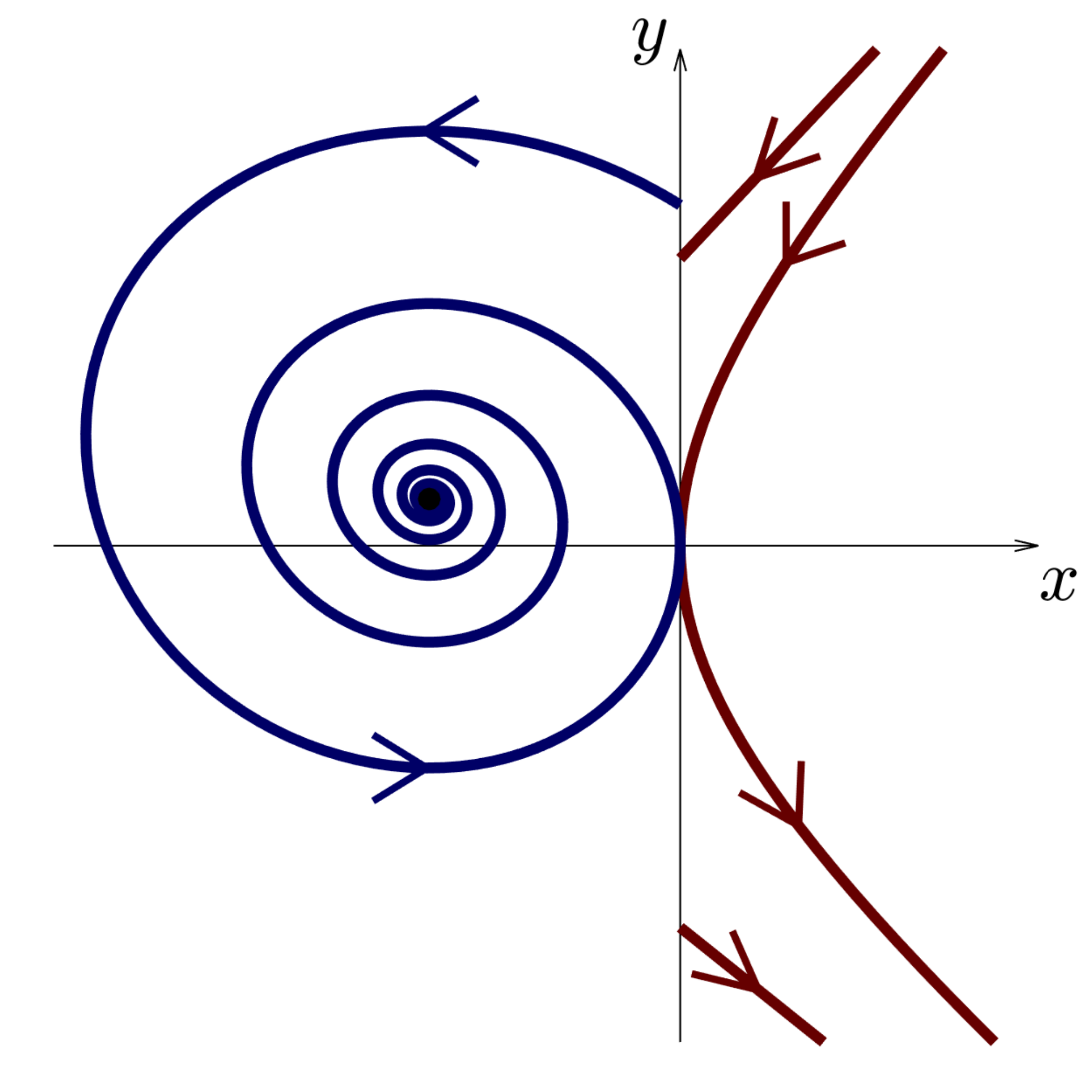}\\
(b)
\end{minipage}
\caption{
Illustrative examples of Remark \ref{rem:estabmonodromico}  regarding the need of assuming the existence of the Poincar\'e half-maps for Proposition \ref{prop:estabmonodromico}. The values of the parameters are $D_L=1$, $D_R=-1$, $a_L=-7/10$, $a_R=1$, and: (a) $T_L=2/10$, $T_R=a_R T_L/a_L=-2/7$; 
(b) $T_L=-2/10$, $T_R=2/7$. In both cases, the right Poincar\'e half-map (corresponding to a virtual saddle) does not exists, $\xi$ vanishes, and the stability of the monodromic equilibrium is thus provided by the sign of $T_L$ instead of $\xi$.
}\label{fig:remark5}
\end{figure}

Notice that when $4D_L-T_L^2>0$ and $4D_R-T_R^2>0$ the infinity is monodromic for the piecewise linear differential system \eqref{cf}. Taking the definition \eqref{displacement} of the displacement function $\delta$ into account, one can see that it is attracting (resp. repelling) provided that  $\delta(y_0)>0$ (resp.  $\delta(y_0)<0$) for $y_0>0$ big enough. The following result characterizes the stability of the infinity when $4D_L-T_L^2>0, 4D_R-T_R^2>0$, and $T_LT_R\ne 0$.

\begin{proposition}
\label{prop:vi-vi}
Consider the differential system \eqref{cf} with $4D_L-T_L^2>0, 4D_R-T_R^2>0$. Define $c_\infty:=T_L\left( T_L^2D_R-T_R^2D_L\right)$. Then, the infinity is 
attracting (resp. repelling) provided that either $T_L>0$ and $T_R>0$, or $T_LT_R<0$ and $c_\infty>0$ (resp. either $T_L<0$ and $T_R<0$, or $T_LT_R<0$ and $c_\infty<0$).   Moreover,  if $T_L T_R<0$, then $\sgn(\delta(y_0))=\sgn(c_{\infty})$ for $y_0$ big enough.
\end{proposition}

\begin{proof} Since $4D_L-T_L^2>0$ and $4D_R-T_R^2>0$, from Proposition \ref{prop:sinf}, it follows that
\[
\lim_{y_0\to +\infty}\frac{y_L(y_0)}{y_R(y_0)}=\lim_{y_0\to +\infty}\frac{y_L(y_0)}{y_0}\cdot\frac{y_0}{y_R(y_0)}=e^{\pi\mu},
\]
where 
\begin{equation}
\label{eq:mu}
\mu=\frac{T_L}{\sqrt{4D_L-T_L^2}}+\frac{T_R}{\sqrt{4D_R-T_R^2}}.
\end{equation}
Therefore, $\sgn(\delta(y_0))=\sgn(\mu)$ for  $y_0$ sufficiently big. In other  words, the infinity is attracting (resp. repelling) provided that $\mu>0$ (resp. $\mu<0$).

When $T_L T_R>0$, then $\sgn{(T_L)}=\sgn{(T_R)}=\sgn{(\mu)}$ and the conclusion follows. 
When $T_L T_R<0$, a direct computation provides $\sgn{(\mu)}=\sgn{(c_\infty)}$ and the proof is finished. 
\end{proof}

\begin{remark}\label{remark:c0c1c2}
This remark is devoted to providing some useful and interesting relationships among the coefficients $c_0,c_1,c_2,$ and $c_\infty$ (provided in propositions \ref{prop:derivdisplafunct} and \ref{prop:vi-vi}, respectively), which will be used later on.

First, these coefficients satisfy the following equalities
\begin{equation}
\label{ecu:relationc0c1c2}
c_0\left(\begin{array}{c}
D_L \\D_R
\end{array}
\right)-c_2\left(\begin{array}{c}
-a_L T_L\\-a_RT_R
\end{array}
\right)+c_1\left(\begin{array}{c}
a_L^2 \\a_R^2
\end{array}
\right)=0
\end{equation}
and
\begin{equation}
\label{c0c1cinfty}
a_La_RT_L^2 \,c_1+a_L^2 a_R\,{c_\infty}=T_L T_R D_L\,c_0.
\end{equation}
This last equality will appear naturally in the proof of the forthcoming Proposition \ref{prop:hyperlimitcycle} and this is the reason why we have preferred using $c_\infty$ rather than $\mu$ to describe the stability of the infinity in Proposition \ref{prop:vi-vi}.

Second, the set of polynomial functions $\{W_L,W_R\}$, with $W_L$ and $W_R$ defined in \eqref{eq:polyL} and \eqref{eq:polyR}, is linearly dependent, that is,
\[
\operatorname{rank}\left( 
\begin{array}{ccc}
D_L & -a_LT_L & a_L^2 \\
D_R & -a_RT_R & a_R^2
\end{array}
\right) \leqslant 1
\]
if, and only if, $c_0=c_1=c_2=0$.

Lastly, the function $F$ defined in \eqref{CeF} can be written as
\[
F(y_0,y_1)=-\det\left( 
\begin{array}{ccc}
1 &  -(y_0+y_1) & y_0y_1 \\
D_L & -a_LT_L & a_L^2 \\
D_R & -a_RT_R & a_R^2
\end{array}
\right).
\]
\end{remark} 

\section{Some results on the existence of limit cycles}\label{sec:condlimitcycle}

In this short section, some results on the existence of limit cycles are given in terms of the parameters of the differential system \eqref{cf}. 
The first result provides some necessary conditions for the existence of limit cycles, which will be useful in the proof of Theorem \ref{main}.

\begin{proposition}
\label{prop:neccond}
Let us consider the values $c_0$, $c_1$, and $c_2$ given in \eqref{eq:c0c1c2}. If the differential system \eqref{cf} has a limit cycle, then the following relationships hold:
\begin{enumerate}[label = {(\alph*)}]
\item $a_L^2+a_R^2\ne 0$.
\item  $T_LT_R<0$.
\item $c_0^2+(c_1c_2)^2\neq0$.
\end{enumerate}
\end{proposition}

\begin{proof} {(a)}
If $a_L=a_R=0$, then the piecewise linear differential system \eqref{cf} is homogeneous. Hence, any positive multiple of an orbit is also an orbit and, consequently, any periodic orbit must be contained in a continuum of periodic orbits.

{(b)} Assume that a crossing periodic orbit exists, that is, there exists $y_0^*\in I$ such that $y_L(y_0^*)=y_R(y_0^*)=y_1^*$. From Proposition  \ref{rm:signoy0+y1}, one has $-\sgn(T_L)=\sgn(y_0^*+y_1^*)=\sgn(T_R)$ and, consequently, $T_LT_R\leq0$.
Thus, again from Proposition  \ref{rm:signoy0+y1} it follows that  
\begin{equation} \label{eq:de9} \sgn(y_0+y_L(y_0))=-\sgn(T_L)=\sgn(T_R)=\sgn(y_0+y_R(y_0)),\end{equation}
for all $y_0\in \operatorname{int}(I)$.
 Note that if $T_L T_R=0$, then \eqref{eq:de9} implies $y_L(y_0)=y_R(y_0)$ for all $y_0\in \operatorname{int}(I)$, which corresponds to a continuum of periodic orbits.
Therefore, the inequality $T_L T_R<0$ holds provided the existence of a limit cycle.

{(c)} Assume that a limit cycle exists and suppose, by reduction to absurdity, that $c_0^2+(c_1c_2)^2=0$ holds.  On the one hand, the existence of the limit cycle implies the existence of an intersection point between the graphs of Poincar\'e half-maps $y_L$ and $y_R$. On the other hand, from items (a), (b), and expression \eqref{ecu:relationc0c1c2}, the equality $c_0^2+(c_1c_2)^2=0$ is equivalent to $c_0=c_1=c_2=0$.
Thus, from Remark \ref{remark:c0c1c2}, the set of polynomial functions $\{W_L,W_R\}$ is linearly dependent. Consequently, the cubic vector fields $X_L$ and $X_R$, provided in \eqref{dy1L} and \eqref{dy1R}, 
have the same orbits. As stated in theorems \ref{teo:defyL}\ref{rm:derivadaL} and \ref{teo:defyR}\ref{rm:derivadaR}, the graph of Poincar\'e half-map $y_L$ is an orbit of $X_L$ and the graph of Poincar\'e half-map $y_R$ is an orbit of $X_R$. Since they coincide at one point
both graphs must be equal. This contradicts the existence of a limit cycle.
\end{proof}

Notice that, by using Green's formula, one can deduce (as it is done, for example, in \cite{FreireEtAl12}) that the existence of a crossing periodic orbit for the differential system \eqref{cf} imposes the inequality $T_LT_R\leqslant 0$. However, in the proof of Proposition \ref{prop:neccond}, for  the sake of completeness, we have preferred to use directly the properties of Poincar\'e half-maps that are obtained from the integral characterization given in \cite{CarmonaEtAl19} and summarized in Section \ref{sec:properties}.

 Propositions \ref{prop:estabmonodromico} and  \ref{prop:vi-vi} give the stability, in the monodromic case, of the unique monodromic singularity and of the infinity, respectively. Then, a simple combination of both results provides a sufficient condition for the existence of limit cycles. 

\begin{corollary}
\label{cor:existence}
Under the assumptions of propositions  \ref{prop:estabmonodromico} and \ref{prop:vi-vi}, a limit cycle of the differential system \eqref{cf} exists provided that 
$\xi c_{\infty}>0.$
\end{corollary}
\begin{proof}
Under the assumptions of propositions \ref{prop:vi-vi} and  \ref{prop:estabmonodromico}, one has that $\sgn(\delta(y_0))=\sgn(c_{\infty})$ for $y_0$ big enough and $\sgn(\delta(y_0))=-\sgn(\xi)$ for $y_0$ sufficiently close to $\lambda_0$.
Then, the existence of a periodic solution follows from the intermediate value theorem, which implies the existence of a zero of $\delta$. The fact that this periodic solution is actually a limit cycle follows from the analyticity of $\delta$, which implies that the zero is isolated.
\end{proof}

As mentioned in the introduction, this corollary is an extension to sewing differential systems of \cite[Proposition 15]{FreireEtAl98}, where conditions for the existence of limit cycles are given for the continuous case.

\section{Uniqueness of hyperbolic limit cycles}\label{sec:hyperlimitcycle}

This section is devoted to showing that if a hyperbolic limit cycle exists for a piecewise linear differential system \eqref{cf}, then it is unique and its stability is determined by the sign of $\xi$ defined in \eqref{eq:xi}.

We know that the curve $\gamma=F^{-1}(\{0\})$, where $F$ is given in \eqref{CeF}, separates the attracting hyperbolic crossing limit cycles from the repelling ones, as stated in Proposition \ref{prop:derivdisplafunct}. In addition, from Remark \ref{rem:hype}, $\gamma$ divides the fourth quadrant $\QQ$ into two connected components. Then, in Proposition \ref{prop:hyperlimitcycle}, the uniqueness of the hyperbolic limit cycles will be obtained, at first under the generic condition $c_0c_1(c_2^2-c_1c_0)\neq0$, by showing that the intersection points between the Poincar\'{e} half-maps which are not located in $\gamma$, if any, are all of them included in a single connected component of $\QQ\setminus\gamma$. By a simple reasoning on the persistence of hyperbolic limit cycles and their stability under small perturbations, the result is immediately extended for all the cases. The proof of Proposition \ref{prop:hyperlimitcycle} is based on Proposition \ref{prop:cortegamma}, which analyzes the intersection between the graph of each Poincar\'e half-map with $\gamma$ in the interior of the fourth quadrant $\QQ$. 

The next result, whose proof follows directly from elementary analysis for functions of one variable, will be of major importance in the proof of propositions \ref{prop:cortegamma} and \ref{prop:hyperlimitcycle}.

\begin{lemma}\label{lem:analysis}
Let $I$ be an interval and  $\eta:I\rightarrow\mathbb{R}$ be a differentiable function. Assume that $\sgn(\eta'(u^*))$ is distinct from zero and the same for every $u^*\in I$ such that $\eta(u^*)=0$. Then, the function $\eta$ has at most one zero in $I$.
\end{lemma}

Remind that under condition $c_1(c_2^2-c_1c_0)\neq0$ (see Remark \ref{rem:hype})  $\gamma$ is a non-degenerate hyperbola in
$\mathbb{R}^2$. The next result 
affirms that, under this condition,
 the graph of each Poincar\'e half-map intersects $\gamma$ at most once in the interior of the fourth quadrant $\Int(\QQ)$ and the intersection, if it exists, is transversal. 

\begin{proposition}
\label{prop:cortegamma} 
Consider the values $c_0,c_1,c_2$ defined in \eqref{eq:c0c1c2} and the set $\gamma$ defined in Remark \ref{rem:hype}. Let us assume that $c_1(c_2^2-c_1c_0)\neq0$. Then, the graph of each Poincar\'e half-map intersects $\gamma$ at most once  in the interior of the fourth quadrant $\Int(\QQ)$ and the intersection, if it exists, is transversal.
\end{proposition}
\begin{proof}
Under the hypothesis $c_1(c_2^2-c_1c_0)\neq0$, 
from Remark \ref{rem:hype},  the set $\gamma$ is a hyperbola and only one of its branches may intersect the fourth quadrant $\QQ$. In this case, the portion of $\gamma$ included in $\Int(\QQ)$, which will be denoted by $\widehat \gamma$, can be written as a graph $y_1=\phi(y_0)$, where $\phi:I_{\phi}\to\R$ is a continuous rational function defined in an open interval $I_{\phi}$ that does not contain the point and $-c_2/c_1$ (the value of the vertical asymptote). 

According to Corollary \ref{bi2quad}, the graph of any Poincar\'e half-map ($y_L$ or $y_R$) 
is either included in the bisector of the fourth quadrant or it does not intersect this bisector except perhaps at the origin. In the first case, the intersection between the graph of the Poincar\'e half-maps and $\gamma$ has already been treated in 
Remark \ref{rem:hype}  such that it only remains to prove the result when the graphs of both Poincar\'e half-maps do not intersect the bisector of the fourth quadrant for $y_0>0$.

In order to analyze the number of intersection points between the graphs of the Poincar\'e half-maps and $\widehat\gamma$, we consider the following difference functions 
\begin{equation}\label{eta}
\eta_L(y_0)=\phi(y_0)-y_L(y_0)\quad\text{and}\quad \eta_R(y_0)=\phi(y_0)-y_R(y_0)
\end{equation}
defined on $I_L^{\eta}:=\Int (I_L)\cap I_{\phi}$ and $I_R^{\eta}:=\Int (I_R)\cap I_{\phi}$, respectively. Hence, the proof of the proposition will follow by showing that each one of the functions $\eta_L$ and $\eta_R$ has at most one zero and, if it exists, it is simple.

If $y_0^*>0$ is such that $\eta_L(y_0^*)=0$ or $\eta_R(y_0^*)=0$, then
\[
\eta_L'(y_0^*)=\dfrac{1}{y_L(y_0^*)W_L(y_0^*)(c_2 +c_1 y_0^*)}\big\langle\nabla F(y_0^*,\phi(y_0^*)),X_{L}(y_0^*,\phi(y_0^*))\big\rangle
\]
or
\[
\eta_R'(y_0^*)=\dfrac{1}{y_R(y_0^*)W_R(y_0^*)(c_2 +c_1 y_0^*)}\big\langle\nabla F(y_0^*,\phi(y_0^*)),X_{R}(y_0^*,\phi(y_0^*))\big\rangle,
\]
respectively. 

Since $\widehat \gamma$ is a portion of a branch of the hyperbola $\gamma$, then $c_2 +c_1 y_0$ has constant sign for $y_0 \in I_{\phi}.$ Also, the existence of Poincar\'e half-maps implies that $W_L(y_0)$ and $W_R(y_0)$ are strictly positive and $y_L(y_0)$ and $y_R(y_0)$ are strictly negative for $y_0\in \Int(I_L)$ and $y_0\in \Int(I_R),$ respectively (see theorems \ref{teo:defyL}\ref{signWL}  and \ref{teo:defyR}\ref{signWL}).  Therefore, the denominators $y_L(y_0)W_L(y_0)(c_2 +c_1 y_0)$ and $y_R(y_0)W_R(y_0)(c_2 +c_1 y_0)$ have constant signs for $y_0\in I_L^{\eta}$ and $y_0\in I_R^{\eta}$, respectively. 

In light of Lemma \ref{lem:analysis},  the number of zeros of $\eta_L$ and $\eta_R$ will be studied by means of the inner products $\langle\nabla F(y_0,y_1),X_{L}(y_0,y_1)\rangle$ and $\langle\nabla F(y_0,y_1),X_{R}(y_0,y_1)\rangle$, for $(y_0,y_1)\in\widehat\gamma$ such that $y_0\in \Int(I_L)$ and $y_0\in \Int(I_R)$, respectively.

The gradient of $F$ is obtained by taking derivatives in expression \eqref{Fe1}. Moreover, from this same expression, 
the equation $F(y_0,y_1)=0$ is equivalent to the relationship $W_L(y_1)W_R(y_0)=W_L(y_0)W_R(y_1)$, because $(y_0,y_1)$ is located in $\Int(\QQ)$ and so $y_0\ne y_1$.
Now, by substituting this last relationship into the inner products we get for $y_0\in I_L^{\eta}$
\begin{equation}\label{GL}
\begin{array}{rl}
G_{L}(y_0):=&\big\langle\nabla F(y_0,y_1),X_{L}(y_0,y_1)\big\rangle\big|_{y_1=\phi(y_0)}
\\
=&W_{L}(y_1)\big(a_LT_LW_R(y_0)-a_RT_R W_L(y_0)\big)\big|_{y_1=\phi(y_0)}\vspace{0.2cm}\\
=& W_{L}(\phi(y_0))(c_0-c_1y_0^2)
\end{array}
\end{equation}
and for $y_0\in I_R^{\eta}$
\begin{equation}\label{GR}
\begin{array}{rl}
G_{R}(y_0):=&\big\langle\nabla F(y_0,y_1),X_{R}(y_0,y_1)\big\rangle\big|_{y_1=\phi(y_0)}
\\
=&W_{R}(y_1)\big(a_LT_LW_R(y_0)-a_RT_R W_L(y_0)\big)\big|_{y_1=\phi(y_0)}\vspace{0.2cm}\\
=& W_{R}(\phi(y_0))(c_0-c_1y_0^2).
\end{array}
\end{equation}

From now on, the argument will be done for $\eta_L$. The same reasoning can be done for $\eta_R$.

On the one hand, assume that $G_{L}$ does not vanish. Then, 
$\sgn(\eta_L'(y_0^*))$ is distinct from zero and coincides for every $y_0^*\in I_L^{\eta}$ such that $\eta_L(y_0^*)=0$. Consequently, from Lemma \ref{lem:analysis} we conclude that $\eta_L$ has at most one zero which is simple.

On the other hand, assume that $G_L$ vanishes. In this case, it vanishes only at $y_0=\widebar{y}_0$ (given in \eqref{eq:puntito}). Notice that $I_L^{\eta}\setminus\{\widebar{y}_0\}=A\cup B,$ where $A$ and $B$ are disjoint intervals and the restricted functions $\eta_L |_{A}$ and $\eta_L |_{B}$ satisfy the hypotheses of Lemma \ref{lem:analysis}. Consequently, each of the restricted functions has at most one zero. In addition, $\eta_L(\widebar{y}_0)\neq0,$ 
otherwise $y_L(\widebar{y}_0)=\phi(\widebar{y}_0)=-\widebar{y}_0$ which cannot happen because the graph of $y_L$ does not intersect the bisector of the fourth quadrant for $y_0>0$.

The remainder of this proof is devoted to showing that $\eta_L$ does not have zeros in $A$ and $B$, simultaneously. In this case, $\eta_L$ will have at most one zero which is simple.

First, it can be seen from Remark \ref{rem:hype} that the function $S_\phi(y_0)=\phi(y_0)+y_0$, $y_0\in I_L^{\eta}$, satisfies $\sgn(S_\phi |_{A})\,\sgn(S_\phi |_{B})<0,$ because the intersection between $\widehat \gamma$ and the bisector of the fourth quadrant only occurs at the point $(\widebar{y}_0,\widebar{y}_1)$,  provided in \eqref{eq:puntito}, and this intersection is transversal. 

Now, suppose, by reduction to absurdity, that there exist $y_a\in A$ and $y_b\in B$ such that $\eta_L(y_a)=\eta_L(y_b)=0.$ Consider the function $S_L(y_0)=y_L(y_0)+y_0$, $y_0\in I_\eta^L.$ Observe that $S_L(y_0)=S_{\phi}(y_0)-\eta_L(y_0)$. Thus, $S_L(y_a)=S_{\phi}(y_a)$ and $S_L(y_b)=S_{\phi}(y_b)$ and, therefore,  $S_L(y_a)S_L(y_b)< 0$. From Bolzano Theorem, there exists $\widecheck y_0\in I_L^{\eta}$ such that $S_L(\widecheck y_0)=0$, that is, $y_L(\widecheck y_0)=-\widecheck y_0$, which contradicts the fact that the graph of $y_L$ does not intersect the bisector of the fourth quadrant for $y_0>0$. It concludes the proof.
\end{proof}

Now, we present and prove the result of uniqueness and stability for hyperbolic limit cycles.

\begin{proposition}
\label{prop:hyperlimitcycle} 
Let be $\xi$ as given in \eqref{eq:xi}.
The differential system \eqref{cf} admits at most one hyperbolic 
limit cycle. If this hyperbolic limit cycle exists, then $\xi\neq0$.
Moreover, it is asymptotically stable (resp.~unstable) provided that $\xi<0$
(resp.~$\xi>0$).
\end{proposition}
\begin{proof}
First of all, notice that it is enough to prove the result under 
the generic condition 
\begin{equation}
\label{equ:GC}
c_0c_1(c_2^2-c_1c_0)\neq0
\end{equation} 
because, 
due to the persistence under small perturbations of hyperbolic limit cycles and their stability, the proof can be immediately extended to the case $c_0c_1(c_2^2-c_1c_0)=0$.

The generic condition provided in \eqref {equ:GC} 
includes the conditions $c_1(c_2^2-c_1c_0)\neq0$ and  $c_0\neq0$. Recall that the first one implies that the zero set of the function $F$ provided in \eqref{CeF}, $\gamma=F^{-1}(\{0\})$, is a non-degenerate hyperbola in $\mathbb{R}^2$ (see Remark \ref{rem:hype}),
which separates the attracting hyperbolic crossing limit cycles from the repelling ones, as stated in Proposition \ref{prop:derivdisplafunct}. In addition, $c_0\neq0$ implies that the hyperbola does not contain the origin.

In order to prove that the differential system \eqref{cf}, under the generic condition \eqref{equ:GC},
has at most one hyperbolic limit cycle and determine its stability, it is necessary to refine the analysis performed in Proposition \ref{prop:cortegamma} about the relative position of $\gamma$ and the curves $y_1=y_{L}(y_0)$ and $y_1=y_{R}(y_0)$.

From Remark \ref{rem:hype}, it is known that hyperbola $\gamma$ split the fourth quadrant $\QQ$ into two disjoint connected sets,  namely
\begin{equation*}\label{Hpm}
H_{\pm}:=\{(y_0,y_1)\in \QQ:\,  \sgn(F(y_0,y_1))=\pm\sgn(c_0)\}.
\end{equation*}
Observe that the connected component  $H_-$ could be the empty set, but $(0,0)\in H_+$ given that $F(0,0)=c_0$. 

As stated in Theorem \ref{teo:defyL}\ref{rm:derivadaL}, the graph of the Poincar\'e half-map $y_L$ is the portion included in $\QQ$ of a particular orbit of the 
cubic vector field $X_L$ provided in \eqref{dy1L}
that evolves forward as $y_0>0$ increases. In a similar way,  from Theorem \ref{teo:defyR}\ref{rm:derivadaR}, the graph of Poincar\'e half-map $y_R$ is the portion included in $\QQ$ of a particular orbit of the 
cubic vector field $X_R$ provided in \eqref{dy1R}
that evolves forward as $y_0>0$ increases. Thus, the signs of the functions $G_L$ and $G_R,$ defined in \eqref{GL} and \eqref{GR}, provide the direction of the intersection between the graphs of the Poincar\'{e} half-maps and $\gamma$.

Since the relative position between $\gamma$ and the origin $(0,0)$ is known, it is natural to conclude our proof by distinguishing the relative positions between the origin and the graphs of the Poincar\'e half-maps, namely: (A) just one of the graphs contains the origin, (B) both of them contain the origin, and  (C) none of them contain the origin. 

From Proposition \ref{prop:neccond}, if the inequality $T_LT_R\geqslant0$ holds, then no limit cycles exist. Accordingly, from now on, it is assumed that $T_LT_R<0$. 
Moreover, the generic condition \eqref{equ:GC} implies that $c_0\neq0$ and then, from \eqref{eq:c0c1c2}, $ a_La_R\neq0$. Therefore, taking Remark \ref{rem:0en0} into consideration, case (A) is equivalent to $a_L a_R>0$, case (B) is equivalent to $a_L>0$ and $a_R<0,$ and  case (C) is equivalent to $a_L<0$ and $a_R>0.$

\medskip

 \noindent{\bf Case (A).} This case corresponds to $a_L a_R>0$. 
Since $T_L T_R<0$, from the second equalities of expressions \eqref{GL} and \eqref{GR}, one obtains that $\sgn(G_{L}(y_0))$ and $\sgn(G_{R}(y_0))$ are constant for $y_0\in I_{\eta}^L$ and $y_0\in I_{\eta}^R$, respectively, and coincide. Here, $I_{\eta}^L$ and $I_{\eta}^R$ are, respectively, the intervals of definition of the functions $\eta_L$ and $\eta_R$ defined in \eqref{eta}.  Thus, since $0$ belongs to the closure of one of the intervals, $I_{\eta}^L$ or $I_{\eta}^R,$ from the third equalities of expressions \eqref{GL} and \eqref{GR}, it follows that
$$\sgn(G_{L}(y_0))=\sgn(c_0)\neq0\quad\text{and}\quad\sgn(G_{R}(y_0))=\sgn(c_0)
\ne 0,$$ for every $y_0\in I_{\eta}^L$ and $y_0\in I_{\eta}^R$, respectively.

This means that, if one of the curves $y_1=y_{L}(y_0)$ or $y_1=y_{R}(y_0)$ intersects $\gamma$, then it crosses $\gamma$ from $H_-$ to $H_+$ as $y_0$ increases. 
In this way, the region $H_+$ may be understood as a trapping region for the graphs of the Poincar\'{e} half-maps as $y_0$ increases (that is, once one of the Poincar\'{e} half-maps enters in this region, cannot leave it as $y_0$ increases).
Since the origin is a point of $H_+$ and the graph of one of the Poincar\'{e} half-maps contains the origin, then this graph is a subset of $H_+$ (see Figure \ref{fig_casoA}).
Hence, if a point $(y_0^*,y_1^*)$ corresponds to a limit cycle, it must be contained in $H_+$. From Proposition \ref{prop:derivdisplafunct}, $\sgn(\delta'(y_0^*))= \sgn(F(y_0^*,y_1^*))=\sgn(c_0)\neq0$. Therefore, from Lemma \ref{lem:analysis}, $\delta$ has at most one zero which provides the uniqueness of the limit cycles.  In addition, since $a_L a_R>0$, the equality $\sgn(c_0)=\sgn(a_RT_L-a_LT_R)=\sgn(\xi)$ holds and the proof is concluded for case (A).

\begin{figure}[H]
    \begin{center}
     \includegraphics[width=7cm]{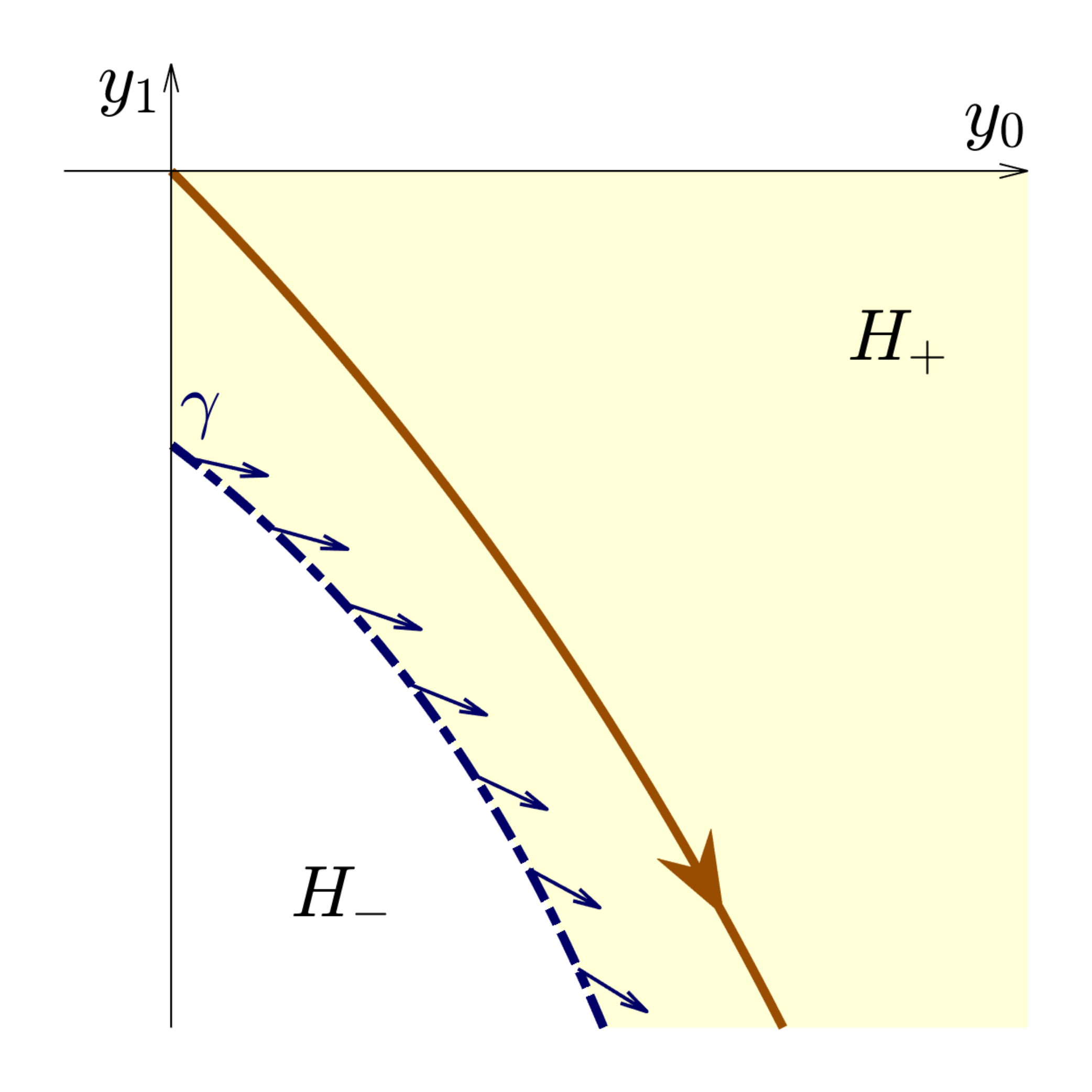}
    \end{center}
     \caption{In case (A), any intersection between the Poincar\'{e} half-maps and $\gamma$ occurs from $H_-$ to $H_+$ as $y_0$ increases. Since $(0,0)\in H_+,$ the Poincar\'{e} half-map starting at the origin do not cross $\gamma$, which implies the uniqueness of hyperbolic limit cycles. The dashed curve represents the curve $\gamma$; the arrows represent the vectors fields $X_L$ and $X_R$, which are proportional over $\gamma$; and the continuous curve represents the  graph of the Poincar\'{e} half-map starting at the origin which is trapped in $H_+$}\label{fig_casoA}
\end{figure}

\medskip

\noindent{\bf Case (B).} In this case, $a_L>0$ and $a_R<0$ and so $\xi c_0<0$. Notice that
\begin{itemize}
\item $y_L(0)=y_R(0)=0$ or, equivalently, $\delta(0)=0$;
\item $F(0,0)=c_0$ and, therefore, $(0,0)\in H_+$;
\item and, from \eqref{eq:deltarrollo}, $\sgn(\delta(y_0))=-\sgn(\xi)=\sgn(c_0)$ for $y_0>0$ sufficiently small.
\end{itemize}
Consequently, there exist $\varepsilon>0$ such that $(y_0,y_L(y_0)),$ $ (y_0,y_R(y_0))\in H_+$ and $ \sgn(\delta(y_0))=\sgn(c_0)$ for every $y_0\in(0,\varepsilon)$.

If the displacement function $\delta$ does not vanish for any $y_0>0$, then no limit cycles exist. Otherwise, there exists 
$\widehat{y}_0>\varepsilon$ such that $\delta(\widehat{y}_0)=0$ and $\xi \delta(y_0)<0$ for any $y_0\in(0,\widehat{y}_0)$. Consequently, $\xi \delta'(\widehat{y}_0)\geqslant0.$  From Proposition \ref{prop:derivdisplafunct},  $\sgn(F(\widehat y_0,y_L(y_0))=\sgn(\delta'(\widehat y_0))$ and, then, taking into account that $\xi c_0<0$, we have that either $\sgn(F(\widehat y_0,y_L(\widehat y_0)))=0$ or $\sgn(F(\widehat y_0,y_L(\widehat y_0)))=\sgn(\xi)=-\sgn(c_0)$. Therefore, the point $(\widehat{y}_0,y_L(\widehat{y}_0))=(\widehat{y}_0,y_R(\widehat{y}_0))$ belongs to $H_-\cup\gamma$. This means that the graphs of both Poincar\'e half-maps, $y_L$ and $y_R$, have intersected $\gamma$ at the points, let us say $(y_0^L,y_1^L)$ and $(y_0^R,y_1^R)$, respectively. Notice that $\widecheck y_0:=\max\{y_0^L,y_0^R\}\leq \widehat y_0$. Hence, from Proposition \ref{prop:cortegamma}, the region $H_-$ may be understood as a trapping region for the graphs of the Poincar\'{e} half-maps as $y_0$ increases (see Figure \ref{fig_casoB}).  Consequently, if $y_0^*>\widecheck y_0$ is such that $\delta(y_0^*)=0$, then $\sgn(\delta'(y_0^*))=-\sgn(c_0)$. From Lemma \ref{lem:analysis}, $\delta$ has at most one zero in $I\cap (\widecheck y_0,+\infty)$. This implies the uniqueness of hyperbolic limit cycles, because $\delta$ does not have simple zeros for $y_0\leq \widecheck y_0.$  In addition, since $a_L a_R<0$, the equality $\sgn(c_0)=-\sgn(a_RT_L-a_LT_R)=-\sgn(\xi)$ holds and the proof is concluded for case (B).

\begin{figure}[H]
    \begin{center}
     \includegraphics[width=7cm]{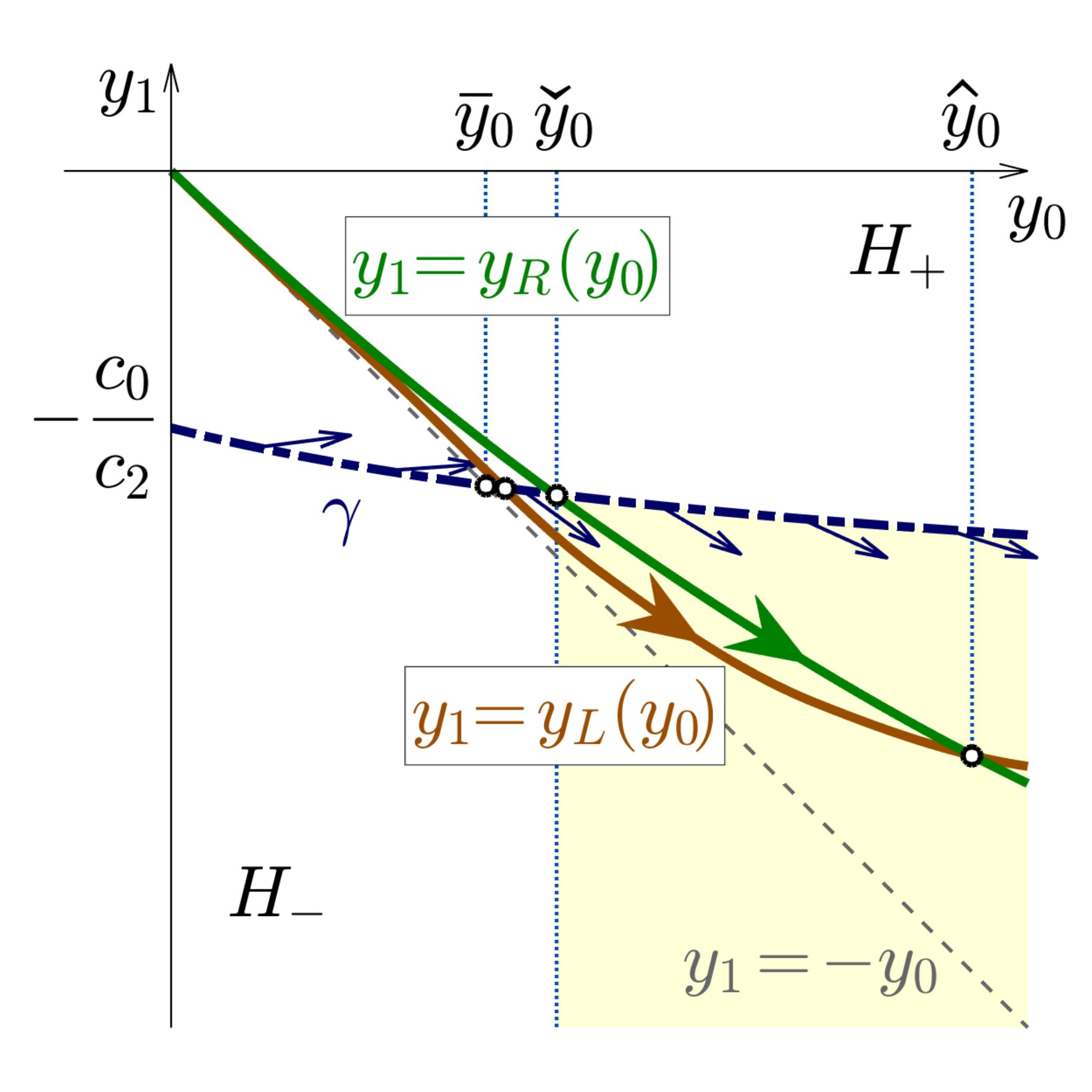}
    \end{center}
     \caption{In case (B), if the displacement function $\delta(y_0)=y_R(y_0)-y_L(y_0)$ vanishes at $\hat y_0$, then the Poincar\'{e} half-maps must intersect $\gamma$. Since such maps intersect $\gamma$ at most  once, they do not intersect $\gamma$ for $y_0>\widecheck y_0$, which implies the uniqueness of hyperbolic limit cycles. The dashed curve represents the curve $\gamma$; the arrows represent the vectors fields $X_L$ and $X_R$, which are proportional over $\gamma$; and the continuous curves represent the graphs of the Poincar\'{e} half-maps which are trapped in the colored region for $y_0>\widecheck y_0$.}\label{fig_casoB}
\end{figure}

\medskip

\noindent{\bf Case (C).} Now, 
$a_L<0$ and $a_R>0$. From \eqref{integralF} and \eqref{integralB}, the existence of the Poincar\'{e} half-maps implies that $4D_L-T_L^2>0$ and $4D_R-T_R^2>0$, so the differential system \eqref{cf} has exactly two focus equilibrium points and
the maps $y_L(y_0)$ and $y_R(y_0)$ tend to $-\infty$ as $y_0$ goes to $+\infty$. Moreover, the infinity is monodromic and, since $T_L T_R<0$, from Proposition \ref{prop:vi-vi}, its stability is characterized by the sign of value $c_\infty$.

Next, it is suitable to consider separately the cases (C1) $c_0c_1<0$ and (C2) $c_0c_1>0$, because they show certain analogies with cases (B) and (A), respectively.

\medskip

\noindent {\bf (C1)} Suppose that $c_0c_1<0$. From expression \eqref{c0c1cinfty} it is immediate that
\[
\frac{a_L}{T_L^2}{c_\infty}c_0=\frac{D_LT_R}{a_La_RT_L}c_0^2-c_1c_0,
\]
which implies that $c_0c_\infty<0$ and so $c_1 c_{\infty}>0$.
Therefore, by Proposition \ref{prop:vi-vi},
the infinity is attracting (resp. repelling) when $c_1>0$ (resp. $c_1<0$).

Now, we perform a change of variables to transform the infinity into the origin for the purpose of applying a similar reasoning to case (B). 
Let us consider the function
\[
\Delta(Y_0)=\left\{
\begin{array}{ccl}
{\displaystyle \frac{1}{y_L(1/Y_0)}- \frac{1}{y_R(1/Y_0)} }& \mbox{if}  & Y_0\ne 0 \,\, \mbox{and} \,\, 1/Y_0\in I, \\
\noalign{\medskip}
0 & \mbox{if}  & Y_0=0,
\end{array}
\right.
\]
where $I$ is the domain of the displacement function $\delta$ provided in \eqref{displacement}. Notice that $\Delta$ is a continuous function at $Y_0=0$ and its first derivative on the right is
$$
  \Delta'(0^+)=\lim_{Y_0\searrow0} \frac{\Delta(Y_0)}{Y_0}=\lim_{y_0\to+\infty} \frac{y_0}{y_L(y_0)}- \frac{y_0}{y_R(y_0)}.
$$
From Proposition \ref{prop:sinf}, $\sgn\left( \Delta'(0^+)\right)=\sgn(\mu)$, where $\mu$ is given in \eqref{eq:mu}. In addition, from the proof of Proposition \ref{prop:vi-vi}, with $T_L T_R<0$, one has that $\sgn\left( \Delta'(0^+)\right)=\sgn(\mu)=\sgn(c_\infty)=\sgn(c_1)$.

If $\Delta(Y_0^*)=0$, with $Y_0^*>0$, then the differential system \eqref{cf} has a  periodic orbit corresponding to $(y_0^*,y_1^*)=(1/Y_0^*,1/Y_1^*)$, being $Y_1^*=1/y_L(1/Y_0^*)=1/y_R(1/Y_0^*)<0$. Moreover, the equality $\sgn(\Delta'(Y_0^*))=-\sgn(\delta'(y_0^*))$ holds and, consequently, the periodic orbit is a hyperbolic  limit cycle provided that 
$\Delta'(Y_0^*)\ne 0$.

By applying the change of variables  $y_0=1/Y_0, y_1=1/Y_1$,
for $Y_0>0$ and $Y_1<0$, to function $F$ given in \eqref{CeF}, we obtain
$$
  F(y_0,y_1)=\frac{c_1+c_0Y_0Y_1+c_2(Y_0+Y_1)}{Y_0Y_1}.
$$
From Proposition \ref{prop:derivdisplafunct}, since  $Y_0^* Y_1^*<0$, then  $\sgn(\Delta'(Y_0^*))=-\sgn(\delta'(y_0^*))=-\sgn(F(y_0^*,y_1^*))=\sgn(\widetilde F(Y_0^*,Y_1^*))$, where
\[
\widetilde F(Y_0,Y_1)=c_1+c_0Y_0Y_1+c_2(Y_0+Y_1).
\]

Now, since $\Delta(0)=0$, $\sgn\left(\Delta'\left(0^+\right)\right)=\sgn(c_1)$, and $\widetilde F(0,0)=c_1\ne0$, one has 
\begin{itemize}
\item $\Delta(0)=0$;
\item $\widetilde F(0,0)=c_1$;
\item and $\sgn(\Delta(Y_0))=\sgn(c_1)$ for $Y_0>0$ sufficiently small.
\end{itemize}
Hence, an analogous reasoning to case (B) provides that 
a hyperbolic limit cycle corresponding to a point $(Y_0^*,Y_1^*)=(1/y_0^*,1/y_1^*)$ satisfies $\sgn(\Delta'(Y_0^*))=-\sgn(c_1)$, therefore it is unique and asymptotically stable (resp.~unstable) provided that $\xi<0$ 
(resp.~$\xi>0$), because $\sgn(\delta'(y_0^*))=-\sgn(\Delta'(Y_0^*))=\sgn(c_1)=-\sgn(c_0)=\sgn(\xi).$

\medskip

\noindent {\bf (C2)} Suppose that $c_0c_1>0$. Without loss of generality, we can assume
that $T_L>0$ and $T_R<0$. Indeed, the case $T_R>0$ and $T_L<0$ can be reduced to the previous one by the change of variables and time rescaling $(t,y)\mapsto (-t,-y)$. 

Since $a_L<0$ and $T_L>0$, from \eqref{integralF},  it follows that $y_L(0)<0$. Analogously, since $a_R>0$ and $T_R<0$, from \eqref{integralB}, it follows that $y_R(0)<0$. 
Thus,  according to Corollary \ref{bi2quad}, the graphs of both Poincar\'e half-maps must be located below the bisector of the fourth quadrant.

Now, we will prove that the graphs of the Poincar\'{e} half-maps are included in $H_-$. From relationship \eqref{ecu:relationc0c1c2}, one obtains
\[
\frac{c_2}{c_1}a_L T_L=-a_L^2-\frac{c_0}{c_1} D_L\quad \mbox{and} \quad \frac{c_2}{c_0}a_RT_R=-\frac{c_1}{c_0}a_R^2-D_R.
\]
Hence, $c_1 c_2>0$ and  $c_0 c_2>0$. As a consequence, from Remark \ref{rem:hype}, the center of the hyperbola $\gamma$ is located at the third quadrant, $\gamma$ intersects the axis $y_0=0$ for the value $y_1=-c_0/c_2<0$ and the bisector of the fourth quadrant at the point given in \eqref{eq:puntito}. In this way, the curve $\gamma$, the axis $y_0=0$,  and the bisector of the fourth quadrant define a bounded region $\Omega\subset H_+$.

From expressions  \eqref{GL} and \eqref{GR}, the orbits of the vector fields $X_L$ and $X_R$, as $y_0$ increases, cross $\gamma$ from $H_-$ to $H_+$ only for $0<y_0<\sqrt{c_0/c_1}$. Accordingly, $\Omega$ is a bounded trapping region as $y_0$ increases for the graphs of the Poincar\'e half-maps. However, since these graphs are unbounded, they cannot enter $\Omega$ and, therefore, they do not intersect $\gamma$  (see Figure \ref{fig_casoC}). This implies that they are included in $H_-$ and the conclusion of case (C2) follows by a similar reasoning to case (A).

\begin{figure}[H]
    \begin{center}
     \includegraphics[width=7cm]{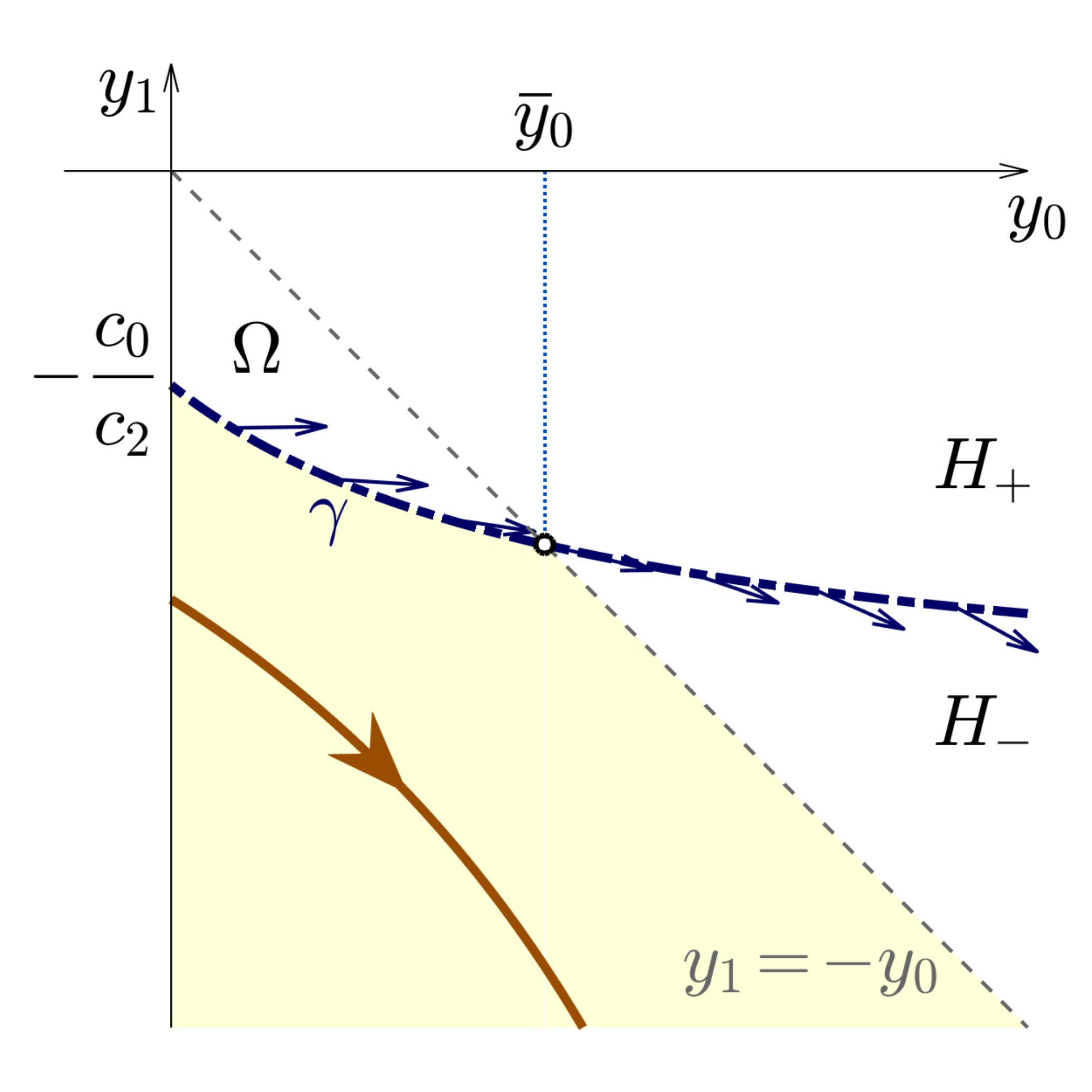}
    \end{center}
     \caption{In case (C2), if the one of the Poincar\'{e} half-maps enters the region $\Omega$ it cannot leave this region. Since $\Omega$ is bounded and, in this case, the Poincar\'{e} half-maps are unbounded, they cannot enter the region $\Omega$. Therefore, such maps do not intersect $\gamma$, which implies the uniqueness of hyperbolic limit cycles. The dashed curve represents the curve $\gamma$; the arrows represent the vectors fields $X_L$ and $X_R$, which are proportional over $\gamma$; and the continuous curve represents one of the graphs of the Poincar\'{e} half-maps which is  trapped in the colored region.}\label{fig_casoC}
\end{figure}
\end{proof}

\section{Nonexistence of degenerate limit cycles} \label{sec:dlc}
From Proposition \ref{prop:hyperlimitcycle}, we have proven that if a hyperbolic limit cycle exists for piecewise linear differential system \eqref{cf}, then it is unique and its stability is determined by the sign of $\xi$ defined in \eqref{eq:xi}. Hence, in this section, we conclude the proof of Theorem \ref{main} by providing the nonexistence of degenerate limit cycles of the differential system \eqref{cf}. 
 
\begin{proposition}\label{prop:dlc}
The differential system \eqref{cf} does not admit degenerate limit cycles.
\end{proposition}
\begin{proof}
Assume,  by contradiction,  that there exists a choice of the parameters $$(a_{L},a_{R},T_{L},T_{R},D_{L},D_{R})=(a_{L}^*, a_{R}^*, T_{L}^*, T_{R}^*, D_{L}^*,D_{R}^*)$$ for which the differential system \eqref{cf} admits a degenerate limit cycle passing through $(0,y_0^*)$ and $(0,y_1^*),$ with $y_1^*<0<y_0^*$, that is $\delta(y_0^*)=\delta'(y_0^*)=0$.
 
 A contradiction will be obtained by showing the existence of a saddle-node bifurcation. That is, we will see that, at $y_0^*$ and $(a_{L}^*, a_{R}^*, T_{L}^*, T_{R}^*, D_{L}^*,D_{R}^*)$, the second derivative of the displacement function $\delta$ with respect to $y_0$ and the first derivative with respect to a parameter are distinct from zero. Consequently, two simple zeros of the displacement function $\delta$ will bifurcate from the degenerate zero $y_0^*$. These zeros correspond to two hyperbolic limit cycles of \eqref{cf}, which contradicts Proposition \ref{prop:hyperlimitcycle}. 

Notice that, from Proposition \ref{prop:neccond}, the inequality $\left(a_L^*\right)^2+ \left(a_R^*\right)^2\neq0$ holds. We shall prove the proposition assuming $a_L^*\ne 0$. An analogous reasoning can be done for the case $a_R^*\ne0$. 

Let $c_0^*$ and $c_2^*$ as defined in \eqref{eq:c0c1c2} for the above fixed parameters.
Since  $\delta(y_0^*)= \delta'(y_0^*)=0$, from
Proposition \ref{prop:derivdisplafunct}, one obtains that 
\[
\sgn(\delta''(y_0^*))=\sgn\left(T_L^* \left( c_2^* y_0^*+c_0^*\right) \right)=\sgn\left(T_R^* \left( c_2^* y_1^*+c_0^*\right) \right),
\]
and so $\sgn(\delta''(y_0^*))\neq0$. Indeed, $y_0^*y_1^*<0$ and, from Proposition \ref{prop:neccond}, $(c_0^*)^2+(c_2^*)^2\neq0$ and $T_L^* T_R^*\neq0$.

Now, depending on the sign of $a_L^*$, we choose either $T_L$ or $a_L$ as the bifurcation parameter in order to unfold two limit cycles.

First, suppose that $a_L^*>0.$ Assume that the parameter $T_L$ is taken in a small neighborhood of  $T_L^*$ and that the other parameters are fixed as $a_{L}=a_{L}^*$, $a_{R}=a_{R}^*$, $T_{R}=T_{R}^*,$ $D_{L}=D_{L}^*,$ and $D_{R}=D_{R}^*.$  Notice that the corresponding displacement function $\delta$, the Poincar\'e half-map $y_L$, and the polynomial function $W_L$  vary with the parameter $T_L$. 
Since $a_L^*> 0$, the forward Poincar\'e half-map $y_L$ is provided by expression \eqref{integralF} as
\[
\int_{y_L(y_0;T_L)}^{y_0}\dfrac{-y}{W_L(y;T_L)}dy=0,
\]
being 
$W_L(y;T_L)=D_L^*y^2-a_L^*T_Ly+\left( a_L^*\right)^2$. 
Therefore,
\[
\frac{\partial y_L}{\partial T_L}(y_0^*;T_L^*)=a_L^*
\frac{W_L(y_1^*;T_L^*)}{y_1^*}
\int_{y_1^*}^{y_0^*}\left(\frac{y}{W_L(y;T_L^*)}\right)^2dy\ne 0.
\]
Since the displacement function \eqref{displacement} satisfies
\[
\frac{\partial\delta}{\partial T_L}\left(y_0^*;T_L^*\right)=-\frac{\partial y_L}{\partial T_L}\left(y_0^*,T_L^*\right)\neq0,\]
  the proof follows for the case $a_L^*>0$.

Finally, suppose that $a_L^*<0$. Assume that the parameter $a_L$ is taken in a small neighborhood of  $a_L^*$ and that the other parameters are fixed as $a_{R}=a_{R}^*$, $T_{L}=T_{L}^*,$ $T_{R}=T_{R}^*,$ $D_{L}=D_{L}^*$, and $D_{R}=D_{R}^*.$ Notice that the corresponding displacement function $\delta$, the Poincar\'e half-map $y_L$, and the polynomial function $W_L$
now vary with the parameter $a_L$.  In this case, the inequality $4D_L^*-\left(T_L^*\right)^2>0$ holds and the expression \eqref{integralF} writes as
\begin{equation}
\label{eq:integralFdege}
\int_{y_L(y_0;a_L)}^{y_0}\dfrac{-y}{W_L(y;a_L)}dy=\frac{2\pi T_L^*}{D_L^*\sqrt{4D_L^*-\left(T_L^*\right)^2}},\end{equation}
being $W_L(y;a_L)=D_L^*y^2-a_LT_L^*y+ a_L^2$.
The changes of variable $Y=y/a_L$ transforms equation \eqref{eq:integralFdege} into the equation
\[
\int_{y_L(y_0;a_L)/a_L}^{y_0/a_L}\dfrac{-Y}{W_L(Y;1)}dY=\frac{2\pi T_L^*}{D_L^*\sqrt{4D_L^*-\left(T_L^*\right)^2}}.
\]
 Thus, it is easy to see that
\[
\frac{\partial y_L}{\partial a_L}(y_0^*;a_L^*)=\frac{\left(y_0^*-y_1^* \right)\left(T_L^*y_0^*y_1^*-a_L^*\left(y_0^*+y_1^* \right) \right)}{y_1^*W_L(y_0^*;a_L^*)}\, .
\]
Observe that $\sgn(T_L^*y_0^*y_1^*)=-\sgn(T_L^*)$ and, from Proposition \ref{rm:signoy0+y1}, $\sgn\left(-a_L^*\left(y_0^*+y_1^* \right)\right)=-\sgn(T_L^*)$. Hence,
\[
\sgn\left(T_L^*y_0^*y_1^*-a_L^*\left(y_0^*+y_1^* \right) \right)=-\sgn\left(T_L^*\right)\ne 0,
\]
which implies that
\[
\frac{\partial\delta}{\partial a_L}\left(y_0^*;a_L^*\right)=-\frac{\partial y_L}{\partial a_L}\left(y_0^*;a_L^*\right)\neq0.
\]
 It concludes this proof.
\end{proof}

\section*{Acknowledgements}

We thank the referees for the helpful comments and suggestions.

The authors thank Espa\c{c}o da Escrita -- Pr\'{o}-Reitoria de Pesquisa -- UNICAMP for the language services provided.

VC is partially supported by the Ministerio de Ciencia, Innovaci\'{o}n y Universidades, Plan Nacional I+D+I cofinanced with FEDER funds, in the frame of the project PGC2018-096265-B-I00. FFS is partially supported by the Ministerio de Econom\'{i}a y Competitividad, Plan Nacional I+D+I cofinanced with FEDER funds, in the frame of the project MTM2017-87915-C2-1-P. VC and FFS are partially supported by the Ministerio de Ciencia e Innovaci\'on, Plan Nacional I+D+I cofinanced with FEDER funds, in the frame of the project PID2021-123200NB-I00. VC and FFS are partially supported by the Consejer\'{i}a de Educaci\'{o}n y Ciencia de la Junta de Andaluc\'{i}a (TIC-0130, P12-FQM-1658). VC and FFS are partially supported by the Consejer\'{i}a de Econom\'{i}a, Conocimiento, Empresas y Universidad de la Junta de Andaluc\'{i}a (US-1380740, P20-01160).
DDN is partially supported by S\~{a}o Paulo Research Foundation (FAPESP) grants 2022/09633-5, 2021/10606-0, 2019/10269-3, and 2018/13481-0, and by Conselho Nacional de Desenvolvimento Cient\'{i}fico e Tecnol\'{o}gico (CNPq) grants 438975/2018-9 and 309110/2021-1.

\section*{Declarations}
\subsection*{Ethical Approval} Not applicable
\subsection*{Competing interests} To the best of our knowledge, no conflict of interest, financial or other, exists.
\subsection*{Authors' contributions} All persons who meet authorship criteria are listed as authors, and all authors certify that they have participated sufficiently in the work to take public responsibility for the content, including participation in the conceptualization, methodology, formal analysis, investigation, writing-original draft preparation and writing-review \& editing.
\subsection*{Availability of data and materials} Data sharing not applicable to this article as no datasets were generated or analyzed during the current study.

\bibliographystyle{abbrv}
\bibliography{references.bib}

\end{document}